\documentclass{amsart}
\usepackage{amssymb,amsthm}

\allowdisplaybreaks
\numberwithin{equation}{section}

\newcommand{\C}{{\mathbb{C}}}
\newcommand{\I}{{\mathfrak{I}}}
\newcommand{\N}{{\mathbb{N}}}
\newcommand{\R}{{\mathbb{R}}}

\newcommand{\Z}{{\mathbb{Z}}}
\newcommand{\op}{\textit{op}}

\newcommand{\mcR}{{\mathcal R}}
\newcommand{\mcM}{{\mathcal M}}

\newcommand{\qtq}[1]{\quad\text{#1}\quad}

\newcommand{\vk}{\varkappa}
\newcommand{\ddt}{\frac{d\ }{dt}}
\DeclareMathOperator{\tr}{tr}
\DeclareMathOperator{\dist}{dist}

\DeclareMathOperator{\supp}{supp}

\let\Im=\undefined\DeclareMathOperator{\Im}{Im}
\let\det=\undefined\DeclareMathOperator{\det}{det}

\newtheorem{theorem}{Theorem}[section]
\newtheorem{prop}[theorem]{Proposition}
\newtheorem{lemma}[theorem]{Lemma}
\newtheorem{corollary}[theorem]{Corollary}
\theoremstyle{definition}
\newtheorem{definition}[theorem]{Definition}
\theoremstyle{remark}
\newtheorem*{remark}{Remark}
\newtheorem*{remarks}{Remarks}

\begin{document}

\title{KdV is wellposed in $H^{-1}$}

\author{Rowan Killip and Monica Vi\c{s}an}

\address
{Rowan Killip\\
Department of Mathematics\\
University of California, Los Angeles, CA 90095, USA}
\email{killip@math.ucla.edu}

\address
{Monica Vi\c{s}an\\
Department of Mathematics\\
University of California, Los Angeles, CA 90095, USA}
\email{visan@math.ucla.edu}

\begin{abstract}
We prove global well-posedness of the Korteweg--de Vries equation for initial data in the space $H^{-1}(\R)$.  This is sharp in the class of $H^{s}(\R)$ spaces.  Even local well-posedness was previously unknown for $s<-3/4$.  The proof is based on the introduction of a new method of general applicability for the study of low-regularity well-posedness for integrable PDE, informed by the existence of commuting flows.   In particular, as we will show, completely parallel arguments give a new proof of global well-posedness for KdV with periodic $H^{-1}$ data, shown previously by Kappeler and Topalov, as well as global well-posedness for the 5th order KdV equation in $L^2(\R)$.

Additionally, we give a new proof of the a priori local smoothing bound of Buckmaster and Koch for KdV on the line.  Moreover, we upgrade this estimate to show that convergence of initial data in $H^{-1}(\R)$ guarantees convergence of the resulting solutions in $L^2_\text{loc}(\R\times\R)$.  Thus, solutions with $H^{-1}(\R)$ initial data are distributional solutions.
\end{abstract}

\maketitle

\section{Introduction}

The Korteweg--de Vries equation
\begin{align}\label{KdV}\tag{KdV}
\ddt q = - q''' + 6qq'
\end{align}
was derived in \cite{KdV1895} to explain the observation of solitary waves in a shallow channel of water.  Specifically, they sought to definitively settle (to use their words) the debate over whether such solitary waves are consistent with the mathematical theory of a frictionless fluid, or whether wave fronts must necessarily steepen.  The equation itself, however, appears earlier; see \cite[p. 77]{Boo}.  The term \emph{solitary wave} has now been supplanted by \emph{soliton}, a name coined in \cite{KruskalZubusky} and inspired by the particle-like interactions they observed between solitary waves in their numerical simulations of \eqref{KdV}.

In a series of papers, researchers at Princeton's Plasma Physics Laboratory demonstrated that equation \eqref{KdV} exhibits a wealth of novel features, including the existence of infinitely many conservation laws \cite{MR0252826} and the connection to the scattering problem for one-dimensional Schr\"odinger equations \cite{GGKM}.  Nowadays, we say that \eqref{KdV} is a completely integrable system (cf. \cite{MR0303132}).

Although we shall focus on mathematical matters here, \eqref{KdV} continues to be an important effective model for a diverse range of physical phenomena; see, for example, the review \cite{MR1329553} occasioned by the centenary of \cite{KdV1895}.

One of the most basic mathematical questions one may ask of \eqref{KdV} is whether it is well-posed.  This is the question of the existence and uniqueness of solutions, together with the requirement that the solution depends continuously on time and the initial data.  As we shall discuss, this topic has attracted several generations of researchers who have successively enlarged the class of initial data for which well-posedness can be shown.  Our principal contribution is the following:

\begin{theorem}[Global well-posedness]\label{T:main}
The equation \eqref{KdV} is globally well-posed for initial data in $H^{-1}(\R)$ or $H^{-1}(\R/\Z)$ in the following sense:  In each geometry, the solution map extends uniquely from Schwartz space to a jointly continuous map $\Phi:\R\times H^{-1}\to H^{-1}$.  Moreover, for each initial data $q\in H^{-1}$, the orbit $\{\Phi(t,q) : t\in\R\}$ is uniformly bounded and equicontinuous in $H^{-1}$.
\end{theorem}

On the circle $\R/\Z$, Schwartz space is coincident with $C^\infty(\R/\Z)$; on the line, it is comprised of those $C^\infty(\R)$ functions that decay (along with their derivatives) faster than any polynomial as $|x|\to\infty$.  For the definition of $H^{-1}(\R)$ and $H^{-1}(\R/\Z)$, see subsection~\ref{S:1.1}; informally, they are comprised of those tempered distributions that are derivatives of $L^2$ functions.  For the precise definition of equicontinuity in the $H^s$ setting, see \eqref{E:equi1}.

The fact that Schwartz-space initial data leads to unique global solutions to \eqref{KdV} that remain in Schwartz class has been known for some time; see, for example, \cite{MR0385355,MR0759907,Sj,MR0410135,MR0261183}.  Indeed, in this class, the solution map is known not only to be continuous, but infinitely differentiable in both variables.  We shall rely on this result in what follows.

In the case of \eqref{KdV} posed on the torus (or equivalently for periodic initial data), Theorem~\ref{T:main} reproduces the principal results of \cite{MR2267286}.  Note that because the circle is compact, uniform boundedness and equicontinuity of the orbit is equivalent to it being pre-compact.  In the line case, one does not expect orbits to be pre-compact; both solitons and radiation preclude tightness from holding globally in time.

The papers \cite{MR2830706,MR2927357} show that well-posedness cannot persist (in either geometry) in $H^{s}$ for any $s<-1$.  In this sense, Theorem~\ref{T:main} is sharp.  On the other hand, one may consider well-posedness at higher regularity $s>-1$.  Existence and uniqueness are immediate from the case $s=-1$; the key is to demonstrate that continuous dependence remains valid in this stronger topology.  In this paper, we will settle the cases left open by prior work, namely, the wellposedness of \eqref{KdV} in $H^{s}(\R)$ with $-1\leq s< -\frac34$.  (On the circle, all $s\geq -1$ were treated already in \cite{MR2267286}.) In fact, the proof of Corollary~\ref{C:2} provides a simple uniform treatment of all $-1\leq s<0$ and adapts trivially to the case of the circle also.

The notion of solution used here (unique limits of Schwartz solutions) coincides with that in \cite{MR2267286} and is informed by several important considerations.  Firstly, as the notion of a solution in the case of Schwartz initial data is firmly settled, any notion of a solution to \eqref{KdV} leading to well-posedness in $H^{-1}$ must produce solutions identical to those given by Theorem~\ref{T:main}.

Secondly, for functions that are merely $C_t H^{-1}_x$ it is not possible to make sense of the nonlinearity in \eqref{KdV} as a space-time distribution in either geometry.  While the local smoothing effect (see subsection~\ref{SS:ls}) provides a potential resolution of this problem in the line setting, there is no natural alternative notion of a weak solution in the circle geometry.  Any methodology that purports to apply in wide generality must adopt a notion of solution that applies in wide generality.

A wider notion of solution was considered in \cite{Christ'05}, namely, limits of smooth solutions in the presence of smooth asymptotically vanishing forcing.  That paper shows (see \cite[\S2.7]{Christ'05}) that with this wider notion of solution, uniqueness cannot be guaranteed for $C_t H^{s}(\R/\Z)$ solutions to \eqref{KdV} already for $s<0$.

From Theorem~\ref{T:main}, we see that the map $\Phi$ is continuous, as was also shown in \cite{MR2267286} for the circle case.  It is natural to ask if this continuity may be expressed more quantitatively.  In some sense, the answer is no: it is shown in \cite{MR2018661} that the data to solution map cannot be uniformly continuous on bounded sets when $s<-\frac34$ in the line case or when $s<-\frac12$ in the circle case.  Nevertheless, the arguments presented here are sufficiently transparent that one may readily obtain information on the modulus of continuity of $q\mapsto\Phi(t,q)$.  Specifically, we find that the key determiners of the modulus of continuity at an initial datum $q\in H^{-1}(\R)$ are the time $t$ in question and the rate at which
$$
\int \frac{|\hat q(\xi)|^2\,d\xi}{\xi^2+4\kappa^2} \to 0\qtq{as} \kappa\to\infty.
$$
Evidently, this integral does not converge to zero uniformly on any open set of initial data.

Let us now turn our attention to a discussion of prior work proving well-posedness for \eqref{KdV}. Discussion of weak solutions (without uniqueness) is postponed until subsection~\ref{SS:ls}.  Our discussion will not be exhaustive; the body of literature on \eqref{KdV} is simply immense.  Nor will we insist on a strict chronology.

Early work on the local and global well-posedness of \eqref{KdV} treated it as a quasi-linear hyperbolic problem.  The appearance of the derivative in the nonlinearity prohibits simple contraction mapping arguments from closing.  The principal methods employed were (i) compactness and uniqueness arguments (e.g. \cite{MR0454425,MR0261183}) combined with parabolic regularization, (ii) convergence of Picard iterates with (e.g. \cite{MR0312097}) or without (e.g. \cite{MR0407477}) parabolic regularization, and (iii) approximation by the Benjamin--Bona--Mahoney (BBM) equation \cite{MR0393887,MR0385355}.

The BBM equation was introduced in \cite{MR0427868}; this equation has a much more regular nonlinearity and global well-posedness was shown there by simple contraction mapping arguments.  The BBM equation has the same Hamiltonian as \eqref{KdV}, namely,
\begin{equation}\label{I HKdV}
H_\text{KdV} (q) := \int \tfrac12 q'(x)^2 + q(x)^3\,dx;
\end{equation}
however the underlying symplectic structure is different.  Of all the prior approaches we know of, the one we follow here is closest in spirit to that of Bona--Smith \cite{MR0385355}, since both we and they employ the idea of approximating the full flow by another Hamiltonian evolution that is more readily controlled.

Incidentally, the problem of local  well-posedness of \eqref{KdV} in Schwartz space, which we shall take for granted here, is rather easier than the works just cited, because one may safely lose regularity in proving continuous dependence of the solution on the initial data.

While multiple authors sought to obtain well-posedness in $H^{s}$ for $s$ as small as possible, these early attempts did not succeed in proving local well-posedness beyond the regime $s>3/2$.  Most significantly, this does not reach the level $s=1$ at which one may upgrade local to  global well-posedness by exploiting conservation of the Hamiltonian.  Nevertheless, global well-posedness was obtained at this time for $s\geq 2$ by using the conserved quantities at such higher regularity discovered in \cite{MR0252826}.

To progress further in this vein, the key has been to exploit the dispersive property of \eqref{KdV}.  Global well-posedness for finite energy initial data on the line was first proved in \cite{MR1086966}, by utilizing local smoothing and maximal function estimates.  The paper actually proves local well-posedness in $H^s(\R)$ for $s>3/4$; the global $H^1(\R)$ result follows trivially from this and conservation of the Hamiltonian.

The next conspicuous benchmark for the well-posedness theory was the treatment of initial data with finite momentum
\begin{equation}\label{I P}
P(q) := \int \tfrac12 q(x)^2 \,dx,
\end{equation}
that is, data in $L^2$.  \emph{Momentum} is the appropriate term here; this quantity is the generator of translations with respect to the standard symplectic structure.  Moreover, this quantity is conserved under the KdV flow.  The \emph{mass} of a wave is given by $\int q(x)\,dx$, which is also conserved, and represents the total deficit (or surplus) of water relative to $q\equiv 0$.

Well-posedness of \eqref{KdV} in $L^2$ was proved both on the line and on the circle by Bourgain in \cite{MR1215780}.  At the heart of this work is the use of $X^{s,b}$ spaces, which efficiently capture the dispersive nature of the equation and effectively control the deviation of the KdV dynamics from solutions to the linear equation $\partial_t q = - q'''$.  After developing suitable estimates in these spaces, the proof proceeds by contraction mapping arguments; thus the solutions constructed depend analytically on the initial data.

Further development and refinement of the methods employed in \cite{MR1215780} ultimately led to a proof of local well-posedness for \eqref{KdV} in $H^{s}(\R)$ for $s\geq -3/4$ and in $H^s(\R/\Z)$ for $s\geq -1/2$.  Excepting the endpoints, this was proved by Kenig--Ponce--Vega in \cite{MR1329387}.  For a discussion of the endpoints, see \cite{MR2018661} and \cite{MR1969209,MR2054622,MR2233689}.  These ranges of $s$ are sharp if one requires the data to solution map to be uniformly continuous on bounded sets; see \cite{MR2018661}.

These local well-posedness results were made global in time in \cite{MR1969209}, excepting the endpoint case $H^{-3/4}(\R)$, which was proved later in \cite{MR2531556,MR2501679}.    At that time, no exact conservation laws were known that were adapted to negative regularity.  To obtain such global results, these authors constructed almost conserved quantities, whose growth in time they were able to control.

While it is true that the conspicuous manifestations of complete integrability of KdV played no particular role in the series of works we have just described, it is difficult to completely decouple these successes from the exact structure of the KdV equation.  In the first place, many of these arguments rely on the absence of unfavorable resonances.  This appears in the multilinear $X^{s,b}$ estimates and (rather more explicitly) in the construction of almost conserved quantities in \cite{MR1969209}.  This is akin to the construction of Birkoff normal form, which may fail due to resonances, but which does succeed in completely integrable systems (cf. \cite{MR0501141,MR2150385}).  As we will discuss below, we now know that KdV admits exact conservation laws adapted to every regularity $s\geq -1$; this offers a rather transparent explanation for the otherwise startling success of \cite{MR1969209} in constructing almost conserved quantities.

The Miura map \cite{MR0252825}, implements a first iteration toward the construction of Birkoff normal form by converting the KdV equation to the mKdV equation, which has a nonlinearity that is one degree higher.  This transformation was one of the first indications that there was something peculiar about \eqref{KdV}.  Moreover, a one-parameter generalization of this transformation, due to Gardner, led to the first proof of the existence of infinitely many polynomial conservation laws; see \cite{MR0252826}.  The Miura map has been very popular in the study of KdV at low regularity.  Most particularly, it allows one to work at positive regularity, where many nonlinear transformations (e.g., pointwise products) are much better behaved.

The breakdown of traditional PDE techniques ultimately stems from a high-high-low frequency interaction that makes it impossible to approximate the KdV flow by a linear evolution even locally in time.  This particular frequency interaction appears in many fluid models, due to the ubiquity of the advection nonlinearity $(u\cdot\nabla) u$, and is exploited crucially in the construction of solutions exhibiting energy growth.

It is worth noting that among the family of monomial gKdV equations, namely, those of the form $\partial_t q = -\partial_x^3 q \pm \partial_x (q^k)$, only for the completely integrable models (i.e., $k=2,3$) does the local well-posedness threshold deviate from scaling.  Indeed, the completely integrable models are \emph{less} well-posed relative to scaling than those with $k\geq 4$.  Ultimately, we see that complete integrability does not completely ameliorate the severity of this nonlinearity when acting on solutions of low regularity.

In this vein, we contend that the complete integrability of a system is not divorced from the class of initial data on which it is studied.  The PDE $\partial_t q = \partial_x q$ posed on the line might immediately be classed as completely integrable; it even belongs to the KdV hierarchy.  However, when the initial data is white-noise, we see that the dynamics is mixing!  On the basis of the results of this paper, we may say that the term \emph{completely integrable} continues to apply to \eqref{KdV} in the class $H^s(\R)$ when $s\geq -1$.

We have not yet explained in what sense \eqref{KdV} can be regarded as completely integrable.  The most common definition applied in finite-dimensional mechanics is that the system has sufficiently many Poisson commuting, functionally independent, conserved quantities.  Here, sufficiently many means half the dimension of the ambient symplectic manifold.  As noted earlier, the fact that \eqref{KdV} admits infinitely many independent conserved quantities was first proved in \cite{MR0252826}.  We have already seen three: the mass, momentum, and energy.  In the original paper, the conservation laws were presented in a microscopic form, that is, as
\begin{equation}\label{micro law}
\partial_t \rho(t,x) + \partial_x j(t,x) = 0,
\end{equation}
where the densities $\rho$ and the currents $j$ are given by particular polynomials in $q$ and its derivatives.  The (macroscopic) conserved quantities are then obtained integrating $\rho$ over the whole line or circle, as appropriate.

The polynomial nature of these conserved quantities is such that, except in the case of $H^\infty$ data (i.e. all derivatives square integrable), all but finitely many of them are infinite.  Moreover, it is also not immediately clear whether these constitute a sufficient number of conserved quantities to call the system completely integrable, even in Schwartz space.  These concerns turn out to be unwarranted.  To explain, we begin with an innovation of Lax \cite{MR0235310}, namely, the introduction of the Lax pair:  Defining
\begin{align*}
L(t) &:= -\partial_x^2 + q(t,x) \qtq{and} P(t) := - 4 \partial_x^3 + 3\bigl(\partial_x q(t,x) + q(t,x) \partial_x \bigr)
\end{align*}
it is easy to verify that
\begin{align*}
\text{$q(t)$ solves \eqref{KdV}} \iff \ddt L(t) = [P(t),\, L(t)].
\end{align*}
As $P(t)$ is always anti-self-adjoint, this shows that at each time slice, the Schr\"odinger operator with potential $q(t,x)$ is unitarily equivalent to that built from the initial data $q(0,x)$.  Speaking loosely, we may say that all spectral properties of $L(t)$ are conserved under the KdV flow.

One of the beauties of the Lax pair is that it works equally well in both geometries.  However, once we try to speak more precisely about which spectral properties are conserved, this unity quickly dissolves.  We will first discuss the periodic case where related ideas have been most successful in tackling the well-posedness problem.

The Schr\"odinger operator on the circle with (periodic) potential $q$ has purely discrete spectrum.  This remains true for potentials that are merely $H^{-1}$ because such perturbations of $-\partial_x^2$ are relatively compact. The Lax pair shows that these (periodic) eigenvalues are then conserved under the flow and so we obtain an infinite sequence of conserved quantities that extend to the case of very low regularity.  There is a direct connection between these eigenvalues and the polynomial conservation laws mentioned earlier; see, for example, \cite[\S3]{MR0397076}.

As it turns out, these eigenvalues are not the most convenient objects for further development.  Rather, one should consider the spectrum of the Schr\"odinger operator associated to the $1$-periodic potential, acting on the whole line.  This set is wholly determined by the periodic eigenvalues; see \cite{MR0749109}.  Nevertheless, this new perspective suggests an alternate set of conserved quantities, namely, the lengths of the gaps in the spectrum.  The virtue of these new quantities can be seen already in the fact that these numbers effectively capture the $H^s$ norms of the potential, at least if $s\geq0$; see \cite{MR0409965}.  While such a priori bounds are useful for well-posedness questions (particularly, to extend solutions globally in time), they do not suffice.

For the purposes of well-posedness, there is no better expression of complete integrability than the existence of action-angle coordinates.  Such coordinates are now known to exist for \eqref{KdV} with data in $H^{-1}(\R/\Z)$ and this result was decisive in the proof of global well-posedness in \cite{MR2267286}.  A key step down this path was the discovery that one should adopt the Dirichlet spectrum (together with the gap lengths) to form a complete set of coordinates and secondly, that these points (which lie in the gaps) should properly be interpreted as lying on the Riemann surface obtained by gluing together two copies of the complex plane cut along the spectrum.  These considerations lead to the definition of angle variables (cf. \cite{MR0427869,MR0427731}) and thence to associated actions \cite{MR0403368}.  A very pedagogical account of these constructions can be found in \cite{MR1997070}; moreover, this monograph culminates in a proof that these variables define \emph{global} action-angle coordinates on each symplectic leaf in the phase-space $L^2(\R/\Z)$.

The proof in \cite{MR2267286} of global well-posedness in $H^{-1}(\R/\Z)$ required two more steps.  The first, carried out in \cite{MR2179653}, was the extension of these coordinates (as a global analytic symplectic diffeomorphism) to $H^{-1}(\R/\Z)$.  The second was to gain adequate control of the frequencies (i.e., the time derivatives of the angles).  Usually, these frequencies are computed as the derivatives of the Hamiltonian with respect to the corresponding actions.  However, the Hamiltonian $H_\text{KdV}$ does not make sense as a function on $H^{-1}(\R/\Z)$!

Let us now turn our attention to the case of \eqref{KdV} posed on the line.  We begin by describing a system of coordinates discovered already in \cite{GGKM} that linearize the flow (at least for a suitable class of data).  While not action-angle variables themselves, such variables can be readily expressed in terms of them; see \cite{MR0303132}.

To the best of our knowledge, the broadest class in which the construction that follows has been successfully completed is $\mathcal L^1_1:= \{q\in L^1(\R): x q(x) \in L^1(\R)\}$; see \cite{MR0897106,MR0792566}.  As we will discuss later, there are compelling reasons to doubt that this construction can be taken much further without substantial new ideas.

Given $q\in \mathcal L^1_1$ and $k\in\C$ with $\Im k\geq 0$, there are unique solutions $f_\pm(x;k)$ to
$$
-f''(x) + q(x)f(x) = k^2 f(x) \qtq{satisfying} f_\pm(x) = e^{\pm ik x} +o(1) \quad\text{as $x\to\pm\infty$.}
$$
These are known as Jost solutions and depend analytically on $k$.  For $k\in\R\setminus\{0\}$, $f_+(x;\pm k)$ are linearly independent solutions to our ODE.  Thus we may define connection coefficients, say $a(k)$ and $b(k)$, so that
\begin{equation}\label{a and b}
f_-(x;k) = a(k) f_+(x;-k) + b(k) f_+(x;k).
\end{equation}
Note that $a(k)$ extends analytically to the upper half-plane, since it can be expressed through the Wronskian of $f_+$ and $f_-$.  There is no such extension of $b(k)$.  The relation to the Wronskian also shows that $a(k)$ has zeros in the upper half-plane precisely at those points $i\kappa_n$ for which $-\kappa_n^2$ is an eigenvalue of the Schr\"odinger operator.

The objects introduced so far do not uniquely characterize the potential $q$.  To do so, one must also consider \emph{norming constants}, $c_n>0$, associated to each eigenvalue $-\kappa_n^2$.  These describe the large-$x$ asymptotics of the $L^2$-normalized eigenfunction $\psi_n(x)$; specifically, $e^{\kappa_n x} |\psi(x)| \to c_n$ as $x\to+\infty$.

As shown already in \cite{GGKM}, the objects just described evolve very simply under \eqref{KdV}: $a(k)$, $|b(k)|$, and the eigenvalues remain constant, while $\arg( b(k))$ and $\log(c_n)$ evolve linearly.  As the forward and inverse scattering problems have been shown to be well-posed in the class $\mathcal L^1_1$, this yields a proof of well-posedness of \eqref{KdV} in this class.  This is the natural analogue of the argument that has proven so successful in the circle geometry.

Unfortunately, well-posedness of the forward/inverse scattering problems (as they are currently understood) begins to break down under very mild relaxations of the condition $q\in \mathcal L^1_1$.  For example, the scattering data fails to determine $q$ already for potentials that are bounded and $O(x^{-2})$ at infinity, due to the presence of zero-energy eigenvalues; see \cite{MR0875319}.  Relaxing our decay restrictions on $q$ to merely $O(|x|^{-1})$ at infinity, gives rise to further problems: positive energy (embedded) eigenvalues may occur (cf. \cite[\S XIII.13]{MR0493421}); moreover, Jost solutions may fail to exist (without WKB correction) at every positive energy.

In \cite{MR2138138}, it is shown that embedded singular continuous spectrum can occur as soon as one passes beyond $O(|x|^{-1})$ decay, even in the slightest.  Moreover, potentials $q\in L^2(\R)$ can yield essentially arbitrary embedded singular spectrum; see \cite{MR2552106}.  The appearance of such exotic spectra leads us to believe that seeking a solution to the well-posedness problem for \eqref{KdV} in $H^{-1}(\R)$ through the inverse scattering methodology has little chance of success at this time.  In particular, we are not aware of any proposal for action-angle variables in such a scenario.  This raises the following question: What other manifestation of complete integrability may hold the key to further progress on the well-posedness problem?

Our answer, in this paper, is the existence of a wealth of commuting flows.  As we will see, the method we propose does not completely supplant PDE techniques, but rather, like the Miura map, provides a new avenue for their application to the KdV problem.  As the existence of an abundance of commuting flows is a necessary (but not sufficient) condition for a system to be considered completely integrable, the method has a good chance of being applicable to any PDE that is considered completely integrable.

The commuting flows associated to the traditional sequence of conserved quantities (based on polynomials in $q$ and its derivatives) are not what we have in mind.  Their well-posedness is at least as difficult as for \eqref{KdV} itself.  Moreover, there is no sense in which they approximate the KdV flow; they are better considered as flowing in orthogonal directions.  Rather, we begin our discussion with
\begin{equation}\label{Intro renorm}
\alpha(\kappa;q) := - \log[a(i\kappa;q)] + \tfrac{1}{2\kappa}\int q(x)\,dx,
\end{equation}
where $a(k;q)$ denotes the coefficient $a(k)$ from \eqref{a and b} associated to the potential $q$.  As noted previously, both $a(k;q)$ and $\int q$ are conserved under the KdV flow; thus one should expect $\alpha(\kappa)$ to also be conserved whenever it is defined.

Unaware that the same idea had already been implemented by Rybkin in \cite{MR2683250}, the authors together with X.~Zhang showed in \cite{KVZ} that $\alpha(\kappa;q)$ is a real-analytic function of $q\in H^{-1}(\R)$, provided $\kappa\geq 1 + 45\|q\|_{H^{-1}}^2$.  We also gave a direct proof that it is conserved for Schwartz initial data.   In these arguments, both we and Rybkin use the fact that $a(k;q)$ can also be written as a Fredholm determinant; see \eqref{O37} below.  That such a determinant representation of this scattering coefficient is possible was first noticed in the setting of three-dimensional central potentials in \cite{MR0044404}.  (See \cite[Proposition~5.4]{MR2154153} or \cite[Lemma~2.8]{MR2310217} for simple proofs in one dimension.)

The renormalization of $\log|a(k)|$ appearing in \eqref{Intro renorm} is essential for considering $q\in H^{-1}(\R)$; without it, one would need to restrict to potentials that are at least conditionally integrable.  Incidentally, this renormalizing term can also be predicted as the leading behaviour of the phase shift via WKB theory.

The goal of the paper \cite{KVZ} was the construction of a variety of low-regularity conservation laws for KdV both on the line and on the circle. (NLS and complex mKdV were also treated there by the same method.) In the case of $H^{-1}(\R)$ bounds for KdV, our argument is essentially that of Rybkin \cite{MR2683250}, who obtained the same result.  Another proof (also independent of Rybkin) can be found in \cite{MR3400442}.  In the line setting, general $H^s(\R)$ bounds for KdV, NLS, and mKdV were obtained, independently, by Koch and Tataru \cite{KT}.  For an earlier partial result, see also \cite{MR3292346}.  In the circle setting, bounds of this type were obtained considerably earlier; see \cite{MR2179653}.

In this paper, we will not rely on the results of \cite{MR3400442,MR2179653,KVZ,KT,MR2683250}. In fact, the proof of Theorem~\ref{T:ls conv} below relies on our development of an alternate argument, which also yields the global $H^{-1}$ bound.  Specifically, we will develop a microscopic version (cf. \eqref{micro law}) of the macroscopic conservation law from \cite{KVZ,MR2683250}.

A priori bounds of the type just described do not in themselves yield well-posedness.  Indeed, conservation of momentum was known already to Korteweg and de~Vries, yet the corresponding well-posedness result did not appear until \cite{MR1215780}.  The key obstacle is always to control differences of solutions.  While individual solutions admit infinitely many conservation laws, the difference of two solutions need not have any.

As discussed previously, the map $q\mapsto \alpha(\kappa;q)$ is analytic; therefore, its derivative (with respect to $q$) is represented (in the sense \eqref{derivative}) by an analytic $H^1$-valued function of $q\in H^{-1}$.  Thus, when we consider the Hamiltonian evolution induced by this functional, namely,
$$
\frac{d\ }{dt} q(t) = \partial_x \frac{\delta \alpha}{\delta q},
$$
we see that the right-hand side is a Lipschitz function on $H^{-1}$ and so well-posedness of this equation follows by the standard ODE argument.  Our ambition (and this appears to be a new idea) is to approximate \eqref{KdV} by this flow.  It turns out that this is possible after one further renormalization, as we will now explain.

It was observed already in \cite{MR0303132} that $\log[a(i\kappa)]$ acts as a generating function for the polynomial conserved quantities; in particular, this yields the asymptotic expansion
$$
\alpha(\kappa;q) = \tfrac{1}{4\kappa^3} P(q) - \tfrac{1}{16\kappa^5} H_\text{KdV}(q) + O(\kappa^{-7}),
$$
using the notations \eqref{I HKdV} and \eqref{I P}.  Inspired by this, one may then postulate that the Hamiltonian
$$
H_\kappa := - 16 \kappa^5 \alpha(\kappa;q) + 4 \kappa^2 P(q)
$$
provides a good approximation to the KdV Hamiltonian for $\kappa$ large.  More ambitiously, one may hope that the KdV flow is well approximated by the flow under $H_\kappa$.  Verifying this and so deducing Theorem~\ref{T:main} occupies the central portion of the paper, namely, Sections~3--5.

Several observations are in order.  Firstly, while $\alpha(\kappa,q)$ is a analytic function on $H^{-1}$, the approximate Hamiltonian $H_\kappa$ is not, because momentum is not.  Nevertheless, well-posedness of the resulting flow is still elementary; see Proposition~\ref{P:H kappa}.

The problem of estimating the discrepancy between the $H_\kappa$ flow and the full KdV flow is much simplified by the fact that the two flows commute.  Indeed, it reduces the question of such an approximation to showing that the flow induced by the difference $H_\text{KdV} - H_\kappa$ is close to the identity for $\kappa$ large and bounded time intervals.

Naturally, one needs to show that this flow is close to the identity in the $H^{-1}$ metric; however, this follows from proximity in much weaker norms, say $H^{-3}$.  The central point here is equicontinuity, or what is equivalent, tightness on the Fourier side; see Lemma~\ref{L:equi 1}.  The equicontinuity of orbits under the flows of interest to us follows from the fact that all conserve $\alpha(\kappa;q)$; see Lemma~\ref{L:equi 2}.  Indeed, from \eqref{alpha as I2}, we see that this functional effectively captures how much of the $H^{-1}$ norm of $q$ lives at frequencies $\xi$ with $|\xi|\gtrsim \kappa$.

One further innovation informs our implementation of the program laid out above, namely, the adoption of the `good unknown' $x\mapsto\kappa - \tfrac{1}{2g(x)}$.   Here $g(x):=g(x;\kappa,q)$ denotes the diagonal of the Green's function associated to the potential $q$ at energy $-\kappa^2$.  For a discussion of this object, see Section~\ref{S:2}.  In particular, it is shown there that the map from $q(x)$ to $\kappa - \tfrac{1}{2g(x)}$ is a real-analytic diffeomorphism, thus, justifying the notion that $g(x)$ may effectively replace the traditional unknown $q(x)$.

Both the diagonal Green's function and its reciprocal appear naturally in several places in our argument, including in the conserved density $\rho$ introduced in \eqref{E:rho defn} and in the dynamics associated to the Hamiltonian $H_\kappa$; see \eqref{H kappa flow q}.   Although our embracement of $g(x)$ is certainly responsible for the simplicity of many of our estimates and concomitantly, for the brevity of the paper, we caution the reader that it is not in itself the key to overcoming the fundamental obstacle confounding previous investigators, namely, the problem of estimating differences between two solutions.

We are not aware of any obstruction to extending the method employed here to a wide range of integrable systems, including those in the AKNS family.  As evidence in favour of this assertion, we demonstrate in Section~\ref{S:periodic} how our method applies in the setting of KdV on the circle.  In Appendix~\ref{S:A} we apply it to the next equation in the KdV hierarchy, following up on an enquiry of a referee. Regarding models in the AKNS family, we note that the functional $\alpha(\kappa,q)$ discussed in \cite{KVZ} is easily seen to have several of the favorable properties needed for our arguments, such as providing global norm control, yielding equicontinuity, and inducing a well-posed Hamiltonian flow.

We do not consider what our results may imply for (real) mKdV via the Miura map.  Rather, it is our hope that our method may soon be adapted
to give an \emph{intrinsic} treatment of the more general \emph{complex} mKdV, which fits within the AKNS family of integrable systems.

Finally, while the ideas presented here are rooted in the complete integrability of KdV, we believe they may prove fruitful beyond this realm.  Specifically, we envision the $H_\kappa$ flow being used as a leading approximation for KdV-like equations in much the same way as the Airy equation, $\partial_t q = -q'''$, has been used as an approximation of KdV itself.

\subsection{Local smoothing}\label{SS:ls}
The local smoothing effect is observed for a wide range of dispersive equations in Euclidean space, both linear and nonlinear.  The underlying physical principle is that when high-frequency components of a wave travel very quickly, they must spend little time in any fixed finite region of space.  Thus, one should expect a gain in regularity locally in space on average in time.  This phenomenon seems to have been first appreciated by Kato, both for linear \cite{MR0190801,MR0234314} and nonlinear \cite{MR0759907} problems.   In \cite{MR0759907}, it is shown that for Schwartz solutions to \eqref{KdV} one has
$$
\int_{-1}^1 \int_{-1}^1 |q'(t,x)|^2 \,dx\,dt \lesssim \| q(0) \|_{L^2}^2 +  \| q(0) \|_{L^2}^6.
$$
This is then used to prove the existence of global weak solutions to \eqref{KdV} for initial data in $L^2(\R)$.  Prior to this, existence of global weak solutions (in either geometry) was known only for data in $H^1$; see \cite{MR0261183}.

In \cite{MR3400442}, Buckmaster and Koch proved the existence of an analogous a priori local-smoothing estimate one degree lower in regularity (on both sides).  This is achieved by using a Miura-type map and adapting Kato's local smoothing estimate for mKdV to the presence of a kink.  This technology is then used to prove the existence of global weak/distributional solutions to \eqref{KdV} with initial data in $H^{-1}(\R)$.  (The nonlinearity may now be interpreted distributionally, because local smoothing guarantees that $q(t,x)$ is locally square integrable in space-time.)  As is usual with the construction of weak solutions, the arguments do not yield uniqueness and continuity in time is only shown with respect to the weak topology.  (Continuous dependence on the initial data is hopeless without first knowing uniqueness.)  For a restricted class of $H^{-1}$ initial data (namely, that in the range of the traditional Miura map), the existence of weak solutions was shown earlier in \cite{MR2189502}; see also \cite{MR0990865}.

In Section~\ref{S:7} we will give a new derivation of the a priori local smoothing bound of \cite{MR3400442}.  Our argument is based on the discovery of a new microscopic conservation law \eqref{E:l5.1h} adapted to regularity $H^{-1}(\R)$, which is then integrated against a suitably chosen weight function.  It is not difficult to extend the a priori bound to the full class of solutions constructed in Theorem~\ref{T:main}.  However, we are able to take the argument one step further and show the following (cf. Proposition~\ref{P:loc smoothing}):

\begin{theorem}\label{T:ls conv} Let $q$ and $\{q_n:n\in\N\}$ be solutions to \eqref{KdV} on the line in the sense of Theorem~\ref{T:main}.  If the initial data obey $q_n(0)\to q(0)$ in $H^{-1}(\R)$, then
\begin{equation}
\iint_K \bigl| q(t,x) - q_n(t,x)\bigr|^2\,dx\,dt \to 0\qtq{as} n\to\infty
\end{equation}
for every compact set $K\subset \R\times\R$.
\end{theorem}

It follows immediately from this result that the solutions we construct are indeed distributional solutions in the line case.

\subsection*{Acknowledgements} R. K. was supported, in part, by NSF grant DMS-1600942 and M. V. by grant DMS-1500707.  We would also like to thank the referee, whose comments and questions led to the inclusion of Appendix~\ref{S:A}.

\subsection{Notation and Preliminaries}\label{S:1.1}
Many of the functions considered in this paper have numerous arguments.  For example, the diagonal Green's function ultimately depends on the location in space $x$, an energy parameter $\kappa$, and the wave profile $q$, which itself depends on time.  We find it advantageous to readability to suppress some of these dependencies from time to time.

We use prime solely to indicate derivatives in $x$; thus $f' =\partial_x f$.

Our conventions for the Fourier transform are as follows:
\begin{align*}
\hat f(\xi) = \tfrac{1}{\sqrt{2\pi}} \int_\R e^{-i\xi x} f(x)\,dx  \qtq{so} f(x) = \tfrac{1}{\sqrt{2\pi}} \int_\R e^{i\xi x} \hat f(\xi)\,d\xi
\end{align*}
for functions on the line and
\begin{align*}
\hat f(\xi) = \int_0^1 e^{- i\xi x} f(x)\,dx \qtq{so} f(x) = \sum_{\xi\in 2\pi\Z} \hat f(\xi) e^{i\xi x}
\end{align*}
for functions on the circle $\R/\Z$.  Concomitant with this, we define
\begin{align*}
\| f\|_{H^{s}(\R)}^2 = \int_\R |\hat f(\xi)|^2 (4+|\xi|^2)^s \,d\xi
\qtq{and}
\| f\|_{H^{s}(\R/\Z)}^2 = \sum_{\xi\in 2\pi\Z} (4+\xi^2)^s |\hat f(\xi)|^2 .
\end{align*}
The use of the number $4$ here rather than the more traditional $1$ has no meaningful effect on these Hilbert spaces (the norms are equivalent); however, this definition simplifies our exposition by making certain key relations exact identities.  More generally, we define
\begin{align*}
\| f\|_{H^{s}_\kappa(\R)}^2 = \int_\R |\hat f(\xi)|^2 (4\kappa^2+|\xi|^2)^s \,d\xi
\qtq{and}
\| f\|_{H^{s}_\kappa(\R/\Z)}^2 = \sum_{\xi\in 2\pi\Z} (4\kappa^2+\xi^2)^s |\hat f(\xi)|^2 .
\end{align*}

Note that $H^{1}_\kappa$ is an algebra in either geometry.  Indeed, one readily sees that
$$
\| f g \|_{H^{1}_\kappa} \lesssim \| f \|_{H^{1}} \| g \|_{H^{1}_\kappa} \leq \| f \|_{H^{1}_\kappa} \| g \|_{H^{1}_\kappa}  \quad\text{uniformly for $\kappa\geq 1$.}
$$
By duality, this implies that
$$
\| f h \|_{H^{-1}_\kappa} \lesssim \| f \|_{H^{\vphantom{+}1}_{\vphantom{\kappa}}} \| h \|_{H^{-1}_\kappa}  \quad\text{uniformly for $\kappa\geq 1$.}
$$

Throughout the paper, we will employ the $L^2$ pairing.  This informs our identification of $H^{-1}$ and $H^1$ as dual spaces and our notation for functional derivatives:
\begin{equation}\label{derivative}
\frac{d\ }{ds}\biggr|_{s=0} F(q+sf) = dF\bigl|_q (f) = \int \frac{\delta F}{\delta q}(x) f(x)\,dx .
\end{equation}

We write $\I_p$ for the Schatten class of compact operators whose singular values are $\ell^p$ summable.  In truth, we shall use the Hilbert--Schmidt class $\I_2$ almost exclusively. When we do use $\I_1$, it will only be as a notation for products of Hilbert--Schmidt operators; see \eqref{I1 from I2}.  Let us quickly recall several facts about the class $\I_2$ that we will use repeatedly:  An operator $A$ on $L^2(\R)$ is Hilbert--Schmidt  if and only if it admits an integral kernel $a(x,y)\in L^2(\R\times\R)$; moreover,
\begin{align*}
\| A \|_{L^2\to L^2} \leq \| A \|_{\I_2} = \iint |a(x,y)|^2\,dx\,dy.
\end{align*}
The product of two Hilbert--Schmidt operators is trace class; moreover,
\begin{align*}
\tr(AB) := \iint a(x,y)b(y,x)\,dy\,dx = \tr(BA) \qtq{and}   |\tr(AB)| \leq \| A \|_{\I_2} \| B \|_{\I_2}.
\end{align*}
Lastly, Hilbert--Schmidt operators form a two-sided ideal in the algebra of bounded operators; indeed,
\begin{align*}
\| B A C \| \leq \| B \|_{L^2\to L^2} \| A \|_{\I_2}\| C \|_{L^2\to L^2}.
\end{align*}
All of this (and much more) is explained very clearly in \cite{MR2154153}.

For the arguments presented here, the problem \eqref{KdV} posed on circle is more favorably interpreted as a problem on the whole line with periodic initial data.  Correspondingly, even in this case, we will be dealing primarily with operators on the whole line, albeit with periodic coefficients.  When we do need to discuss operators acting on the circle $\R/\Z$ in connection with prior work, these will be distinguished by the use of calligraphic font.

\section{Diagonal Green's function}\label{S:2}

The goal of this section is to discuss the Green's function $G(x,y)$ associated to the whole-line Schr\"odinger operator
$$
L := -\partial_x^2 + q
$$
for potentials
\begin{align}\label{B delta}
q \in B_\delta := \{ q \in H^{-1}(\R) : \| q \|_{H^{-1}(\R)} \leq \delta\}
\end{align}
and $\delta$ small.  Particular attention will be paid to the diagonal $g(x):=G(x,x)$ and its reciprocal $1/g(x)$; the latter appears in the energy density associated to the key microscopic conservation law for KdV.

Let us briefly recall one key fact associated to the Schr\"odinger operator with $q\equiv 0$: The resolvent
\begin{align}\label{R resolvent}
R_0(\kappa) = (-\partial^2_x + \kappa^2)^{-1} \qtq{has integral kernel} G_0(x,y;\kappa) = \tfrac{1}{2\kappa} e^{-\kappa|x-y|}
\end{align}
for all $\kappa>0$.

\begin{prop}\label{P:sa L}
Given $q\in H^{-1}(\R)$, there is a unique self-adjoint operator $L$ associated to the quadratic form
$$
\psi \mapsto \int |\psi'(x)|^2 + q(x) |\psi(x)|^2\,dx \qtq{with domain} H^1(\R).
$$
It is semi-bounded.  Moreover, for $\delta\leq \frac12$ and $q\in B_\delta$, the resolvent is given by the norm-convergent series
\begin{align}\label{E:R series}
R := (L+\kappa^2)^{-1} = \sum_{\ell=0}^\infty (-1)^\ell \sqrt{R_0} \Bigl( \sqrt{R_0}\,q\,\sqrt{R_0}\Bigr)^\ell \sqrt{R_0}
\end{align}
for all $\kappa\geq 1$.
\end{prop}

\begin{proof}
The key estimate on which all rests is the following:
\begin{align}\label{R I2}
\Bigl\| \sqrt{R_0}\, q\, \sqrt{R_0} \Bigr\|^2_{\text{op}} \leq \Bigl\| \sqrt{R_0}\, q\, \sqrt{R_0} \Bigr\|^2_{\mathfrak I_2(\R)} &= \frac1\kappa \int \frac{|\hat q(\xi)|^2}{\xi^2+4\kappa^2}\,d\xi .
\end{align}
For $q\in \mathcal S(\R)$ the Hilbert--Schmidt norm can be evaluated directly using \eqref{R resolvent}:
\begin{align*}
\Bigl\| \sqrt{R_0}\, q\, \sqrt{R_0} \Bigr\|^2_{\mathfrak I_2(\R)} &= \frac1{4\kappa^2} \iint q(x) e^{-2\kappa|x-y|}q(y)\,dx\,dy = \text{RHS\eqref{R I2}}.
\end{align*}
This then extends to all $q\in H^{-1}(\R)$ by approximation.

From \eqref{R I2}, we see that
\begin{align*}
\int q(x)|\psi(x)|^2 \,dx \leq \kappa^{-1/2} \|q\|_{H^{-1}} \int |\psi'(x)|^2 + \kappa^2 |\psi(x)|^2\,dx \qtq{for all} \psi \in H^1(\R),
\end{align*}
at least for all $\kappa\geq 1$.   (Note that the LHS here should be interpreted via the natural pairing between $H^{-1}$, which contains $q$, and $H^1$, which contains $|\psi|^2$.)  This estimate shows that
$q$ is an infinitesimally form-bounded perturbation of the case $q\equiv 0$ and so the existence and uniqueness of $L$ follows from \cite[Theorem X.17]{MR0493420}.

In view of \eqref{R I2}, the series \eqref{E:R series} converges provided we just choose $\delta <1$.
\end{proof}

\begin{prop}[Diffeomorphism property]\label{P:diffeo} There exists $\delta>0$ so that the following are true for all $\kappa\geq 1${\upshape:}\\
(i) For each $q\in B_\delta$, the resolvent $R$ admits a continuous integral kernel $G(x,y;\kappa,q)${\upshape;} thus, we may unambiguously define
\begin{align}\label{g defn}
g(x;\kappa,q):=G(x,x;\kappa,q).
\end{align}
(ii) The mappings
\begin{align}\label{diffeos}
q\mapsto g-\tfrac1{2\kappa} \qtq{and} q\mapsto \kappa-\tfrac1{2g}
\end{align}
are (real analytic) diffeomorphisms of $B_\delta$ into $H^1(\R)$.\\
(iii) If $q(x)$ is Schwartz, then so are $g(x)- \tfrac1{2\kappa}$ and $\kappa-\tfrac1{2g(x)}$.  Indeed,
\begin{align}\label{g stronger mapping}
\|g'(x)\|_{H^{s}}  \lesssim_s \| q \|_{H^{s-1}} \qtq{and} \| \langle x\rangle ^s g'(x) \|_{L^2}  \lesssim_s \| \langle x\rangle ^s q \|_{H^{-1}}
\end{align}
for every integer $s\geq 0$.
\end{prop}

\begin{remark}
The diffeomorphism property is necessarily restricted to a neighborhood of the origin because for $q$ large, the spectrum of $L$ may intersect $-\kappa^2$.
\end{remark}

\begin{proof}
Initially, we ask that $\delta\leq \frac12$; later, we will add further restrictions.

From \eqref{E:R series} and \eqref{R I2}, we see that
\begin{align*}
\Bigl\| \sqrt{\kappa^2-\partial_x^2} \bigl(R - R_0\bigr) \sqrt{\kappa^2-\partial_x^2} \Bigr\|_{\I_2} < \infty \qtq{for all} q\in B_\delta \qtq{and all} \kappa\geq 1.
\end{align*}
Consequently, $G-G_0$ exists as an element of $H^1(\R)\otimes H^1(\R)$.  Here we mean tensor product in the Hilbert-space sense (cf. \cite{MR0493419}); note that $H^1(\R)\otimes H^1(\R)$ is comprised of those $f\in H^1(\R^2)$ for which $\partial_x\partial_yf \in L^2(\R^2)$.  It follows that $G(x,y;\kappa,q)$ is a continuous function of $x$ and $y$ and we may define
\begin{align}\label{E:g series}
g(x) = g(x;\kappa,q) = \tfrac1{2\kappa} + \sum_{\ell=1}^\infty (-1)^\ell \Bigl\langle\sqrt{R_0}\delta_x,\ \Bigl( \sqrt{R_0}\,q\,\sqrt{R_0}\Bigr)^\ell \sqrt{R_0} \delta_x\Bigr\rangle,
\end{align}
where inner products are taken in $L^2(\R)$.  This settles (i).

Next we observe that by \eqref{E:R series} and \eqref{R I2},
\begin{align*}
\Bigl| \int f(x) \bigl[g(x) -\tfrac1{2\kappa}\bigr]\,dx \Bigr| &\leq \sum_{\ell=1}^\infty  \Bigl\| \sqrt{R_0}\, f\, \sqrt{R_0} \Bigr\|_{\mathfrak I_2(\R)}\Bigl\| \sqrt{R_0}\, q\, \sqrt{R_0} \Bigr\|_{\mathfrak I_2(\R)}^\ell \\
&\leq 2\delta\kappa^{-1} \|f\|_{H^{-1}(\R)}
\end{align*}
for any Schwartz function $f$.  Thus $g-\frac1{2\kappa}\in H^1(\R)$; indeed,
\begin{align}\label{g H1 bound}
  \bigl\| g - \tfrac1{2\kappa} \bigr\|_{H^1(\R)} \leq 2\delta\kappa^{-1}.
\end{align}
Moreover, this argument precisely shows the convergence of the series \eqref{E:g series} and so that the mapping from $q\in B_\delta$ to $g-\frac1{2\kappa}\in H^{1}(\R)$ is real analytic.

Given $f\in H^{-1}(\R)$, the resolvent identity implies
\begin{align}\label{O35pp}
\frac{d\ }{ds}\biggr|_{s=0} g(x;q+sf) = - \int G(x,y)f(y)G(y,x)\,dy .
\end{align}
In particular, by \eqref{R resolvent},
$$
dg\bigr|_{q\equiv 0} = - \kappa^{-1} R_0(2\kappa),
$$
which is an isomorphism of $H^{-1}_\kappa$ onto $H^1_\kappa$, with condition number equal to $1$. Moreover, by \eqref{R I2}, \eqref{E:R series}, and duality,
\begin{align}\label{inverse input}
\bigl\| dg\bigr|_{q\equiv 0} - dg\bigr|_q \bigr\|_{H^{-1}_\kappa \to H^{1\vphantom{+}}_\kappa} \lesssim \kappa^{-1} \|q\|_{H^{-1}_\kappa}
    \lesssim  \delta \Bigl\| \bigl(dg\bigr|_{q\equiv 0} \bigr)^{-1} \Bigr\|_{H^{1\vphantom{+}}_\kappa \to H^{-1}_\kappa}^{-1} .
\end{align}
Thus choosing $\delta$ sufficiently small, the inverse function theorem guarantees that
\begin{align}\label{g is diffeo}
q \mapsto g - \tfrac{1}{2\kappa} \quad\text{is a diffeomorphism of $\{ q : \|q\|_{H^{-1}_\kappa} \leq \delta\}$ into $H^{1}_\kappa$}.
\end{align}
Note that \eqref{inverse input} combined with the standard contraction-mapping proof of the implicit function theorem guarantees that $\delta$ can be chosen independently of $\kappa$.  The claimed $H^{-1}\to H^1$ diffeomorphism property of this map then follows since
$$
\|q\|_{H^{-1}_\kappa} \leq \|q\|_{H^{-1}} \qtq{and} \| f\|_{H^1_\kappa} \lesssim_\kappa \|f\|_{H^1}.
$$

Choosing $\delta$ even smaller if necessary, \eqref{g H1 bound} together with the embedding $H^1\hookrightarrow L^\infty$ guarantees that
$$
\tfrac1{4\kappa} \leq g(x) \leq \tfrac{3}{4\kappa} \quad\text{for all $q\in B_\delta$.}
$$
Consequently, the second mapping in \eqref{diffeos} is also real-analytic.  To prove that it is a diffeomorphism (for some $\kappa$-independent choice of $\delta$), we simply note that
$$
f \mapsto \frac{f}{1+f}
$$
is a diffeomorphism from a neighbourhood of zero in $H^1(\R)$ into $H^1(\R)$, write
$$
\kappa-\tfrac1{2g} = \kappa \frac{2\kappa(g-\frac1{2\kappa} )}{1+2\kappa(g-\frac1{2\kappa} )},
$$
and use \eqref{g H1 bound} together with \eqref{g is diffeo}.

We now turn our attention to part (iii).  The Green's function associated to a translated potential is simply the translation of the original Green's function.  Correspondingly,
\begin{align}\label{translation identity}
g(x+h;q) = g\bigl(x; q(\cdot+h)\bigr)  \qquad\text{for all $h\in\R$.}
\end{align}
Differentiating with respect to $h$ at $h=0$ and invoking \eqref{E:g series} yields
\begin{align*}
\int \bigl[\partial_x^s g(x)\bigr] f(x)\,dx &\leq \sum_{\ell=1}^\infty \sum_{\sigma} \binom{s}{\sigma}  \Bigl\| \sqrt{R_0}\, f\, \sqrt{R_0} \Bigr\|_{\mathfrak I_2(\R)}
    \prod_{k=1}^\ell \Bigl\| \sqrt{R_0}\, q^{(\sigma_k)}\, \sqrt{R_0} \Bigr\|_{\mathfrak I_2(\R)}.
\end{align*}
Here, the inner sum extends over multi-indices $\sigma=(\sigma_1,\ldots,\sigma_\ell)$ with $|\sigma|=s$.  Maximizing over unit vectors $f\in H^{-1}$, exploiting \eqref{R I2}, and using
$$
 \prod_{k=1}^\ell \bigl\| q^{(\sigma_k)} \bigr\|_{H^{-1}} \leq  \bigl\| q^{(s)} \bigr\|_{H^{-1}} \bigl\| q \bigr\|_{H^{-1}}^{\ell-1},
$$
which is merely an application of Holder's inequality in Fourier variables, this yields
\begin{align*}
\bigl\|\partial_x^s g(x)\bigr\|_{H^{1}}  &\leq  \sum_{\ell=1}^\infty \ell^s \bigl\| q^{(s)} \bigr\|_{H^{-1}} \delta^{\ell-1} \lesssim_s \bigl\| q\bigr\|_{H^{s-1}}.
\end{align*}
Thus we have verified the first claim in \eqref{g stronger mapping}.

To address the second assertion in \eqref{g stronger mapping}, we first make the following claim: For every integer $s\geq 0$,
\begin{align}\label{R langle commutator}
\langle x\rangle^s R_0 = \sum_{r=0}^s \sqrt{R_0}\, A_{r,s} \sqrt{R_0}\, \langle x\rangle^r
\qtq{with operators} \| A_{r,s} \|_{L^2\to L^2}\lesssim_s 1.
\end{align}
This is easily verified recursively, by repeatedly using the following commutators:
\begin{equation}\label{basic commutators}
\begin{aligned}
\bigl[ \langle x\rangle,\  R_0\bigr] = R_0 \bigl[ -\partial_x^2 +\kappa^2 ,\ \langle x\rangle \bigr] R_0 &= -R_0 \bigl(\tfrac{x}{\langle x\rangle} \partial_x + \partial_x\tfrac{x}{\langle x\rangle}\bigr) R_0 \\
\bigl[ \tfrac{x}{\langle x\rangle} \partial_x + \partial_x\tfrac{x}{\langle x\rangle},\ \langle x\rangle\bigr]  &= 2\tfrac{x^2}{\langle x\rangle^2}.
\end{aligned}
\end{equation}

In connection with \eqref{R langle commutator}, let us also pause to note that
\begin{align}\label{E:weight change}
\bigl\| \langle x\rangle^r q \bigr\|_{H^{-1}} \lesssim_{s} \bigl\| \langle x\rangle^s q \bigr\|_{H^{-1}} \quad\text{for any pair of integers $0\leq r\leq s$},
\end{align}
since $\langle x\rangle^{-1}\in H^1(\R)$, which is an algebra.

By applying \eqref{E:g series}, \eqref{R I2}, \eqref{R langle commutator}, and \eqref{E:weight change}, we deduce that
\begin{align*}
\int  f(x) & \langle x\rangle^s \bigl[g(x) -\tfrac{1}{2\kappa} \bigr]\,dx \\
&\lesssim_s \sum_{\ell=1}^\infty \sum_{r=0}^s \Bigl\| \sqrt{R_0}\, f\, \sqrt{R_0} \Bigr\|_{\mathfrak I_2(\R)}  \Bigl\| \sqrt{R_0}\, \langle x\rangle^r q\, \sqrt{R_0} \Bigr\|_{\mathfrak I_2(\R)} \delta^{\ell-1} \\
&\lesssim_s \bigl\| f \bigr\|_{H^{-1}} \bigl\| \langle x\rangle^s q \bigr\|_{H^{-1}}.
\end{align*}
Optimizing over $f\in H^{-1}(\R)$, it then follows that
$$
\| \langle x\rangle^s g'(x) \|_{L^2(\R)} \lesssim_s \bigl\| \langle x\rangle^s \bigl[g(x) -\tfrac{1}{2\kappa} \bigr] \bigr\|_{H^1(\R)} \lesssim_s \| \langle x\rangle ^s q \|_{H^{-1}},
$$
thereby completing the proof of \eqref{g stronger mapping} and so the proof of the proposition.
\end{proof}

\begin{prop}[Elliptic PDE]\label{P:elliptic}
The diagonal Green's function obeys
\begin{align}
g'''(x) &= 2 \bigl[q(x) g(x) \bigr]' + 2 q(x) g'(x) + 4\kappa^2 g'(x) .\label{E:l5.1a}
\end{align}
\end{prop}

\begin{proof}
By virtue of being the Green's function,
\begin{align*}
\bigl( -\partial_x^2 + q(x) \bigr) G(x,y) = -\kappa^2 G(x,y) + \delta(x-y) = \bigl( -\partial_y^2 + q(y) \bigr) G(x,y)
\end{align*}
and consequently,
\begin{align*}
\bigl( \partial_x + \partial_y\bigr)^3 G(x,y) &= \bigl(q'(x)+q'(y)\bigr)G(x,y) + 2\bigl(q(x)+q(y)\bigr) \bigl( \partial_x + \partial_y\bigr) G(x,y) \\
&\qquad - \bigl(q(x)-q(y)\bigr) \bigl( \partial_x - \partial_y\bigr) G(x,y) +4\kappa^2\bigl(\partial_x + \partial_y\bigr) G(x,y).
\end{align*}
Thus specializing to $y=x$, we deduce that
$$
g'''(x) = 2 q'(x) g(x) + 4 q(x) g'(x) + 4 \kappa^2 g'(x),
$$
which agrees with \eqref{E:l5.1a} after regrouping terms.
\end{proof}

\begin{remark}
As will be discussed in the proof of Lemma~\ref{L:D 1/g}, the Green's function can be expressed in terms of two solutions $\psi_\pm(x)$ to the Sturm--Liouville equation (the Weyl solutions); see \eqref{G from psi}.
In this sense, $g(x)=\psi_+(x)\psi_-(x)$ was seen to obey \eqref{E:l5.1a} already in \cite{Appell}.
\end{remark}

\begin{prop}[Introducing $\rho$]\label{P:Intro rho}
There exists $\delta>0$ so that
\begin{align}\label{E:rho defn}
\rho(x;\kappa,q) := \kappa - \tfrac{1}{2g(x)} + \tfrac12\int e^{-2\kappa|x-y|} q(y)\,dy
\end{align}
belongs to $L^1(\R) \cap H^1(\R)$ for all $q\in B_\delta$ and $\kappa\geq 1$.  Moreover, fixing $x\in\R$, the map $q\mapsto \rho(x)$ is non-negative and convex.  Additionally,
\begin{align}\label{E:alpha defn}
\alpha(\kappa;q) := \int_\R \rho(x)\,dx
\end{align}
defines a non-negative, real-analytic, strictly convex function of $q\in B_\delta$, and satisfies
\begin{align}\label{alpha as I2}
\alpha(\kappa;q) \approx \frac{1}{\kappa} \int_\R \frac{|\hat q(\xi)|^2\,d\xi}{\xi^2+4\kappa^2},
\end{align}
uniformly for $q\in B_\delta$ and $\kappa\geq 1$.  Lastly,
\begin{align}\label{O37}
\alpha(\kappa;q) = - \log\det_2\left( 1+ \sqrt{R_0}\, q \, \sqrt{R_0} \right).
\end{align}
\end{prop}

\begin{remarks}
1. Although we shall have no use for the strict convexity of $q\mapsto\alpha(\kappa;q)$ in this paper, it does have important consequences.  Most notably, by the Radon--Riesz argument, it shows that weakly continuous solutions conserving $\alpha(\kappa)$ are automatically norm-continuous.

2. As noted in the Introduction (see \eqref{Intro renorm} and subsequent discussion), the quantity $\alpha(\kappa;q)$ is essentially the logarithm of the transmission coefficient and so well-studied.  Nevertheless, none of the literature we have studied contains the representation \eqref{E:alpha defn} in terms of the reciprocal of the Green's function.  Rather, prior works employ an integral representation based on the logarithmic derivative of one of the Jost solutions; see \eqref{E:a from Weyl}.  To the best of our knowledge, this approach originates in \cite[\S3]{MR0303132}, where it was shown to be an effective tool for deriving polynomial conservation laws and for demonstrating that these polynomial conservation laws appear as coefficients in the asymptotic expansion of the logarithm of the transmission coefficient as $\kappa\to\infty$.
\end{remarks}

Before turning to the proof of Proposition~\ref{P:Intro rho}, we first explain the meaning of RHS\eqref{O37} and then present two lemmas that we shall need.

The symbol $\det_2$ denotes the renormalized Fredholm determinant introduced by Hilbert in \cite{Hilbert}; see \cite{MR2154153} for a more up-to-date exposition.  In the context of Proposition~\ref{P:Intro rho}, our choice of $\delta$ guarantees that the operator
$$
A = \sqrt{R_0}\, q \, \sqrt{R_0} \qtq{obeys} \| A \|_{\I_2} < 1.
$$
Consequently, it suffices for what follows to exploit only the notion of the trace of an operator (rather than determinant) thanks to the identity
\begin{align}\label{det series}
-\log \det_2 \bigl(1 + A \bigr) = \tr\big( A - \log(1+A) \bigr) =  \sum_{\ell=2}^\infty \frac{(-1)^\ell}{\ell} \tr\bigl( A^\ell \bigr).
\end{align}
We shall not delve deeply into such matters here, since \eqref{O37} has no bearing on the proof of well-posedness for KdV; indeed, our only reason for verifying this identity is to make the link to the prior works \cite{KVZ,MR2683250}, which might otherwise seem unrelated.

\begin{lemma}\label{L:D 1/g}
There exists $\delta>0$ so that
\begin{align}\label{GgG identity}
\int \frac{G(x,y;\kappa,q)G(y,x;\kappa,q)}{2g(y;\kappa,q)^2}\,dy = g(x;\kappa,q)
\end{align}
for all $q\in B_\delta$ and all $\kappa\geq 1$.
\end{lemma}

\begin{remark}
Augmenting the proof below with the results of \cite[\S8.3]{MR0069338} shows that \eqref{GgG identity} holds also in the case of $q\in H^{-1}(\R/\Z)$ and $\kappa\geq 1$ that obey \eqref{periodic smallness}.  As below, one first uses analyticity to reduce to a case where one may apply ODE techniques, more specifically, to the case of small smooth periodic potentials.
\end{remark}

\begin{proof}
We choose $\delta>0$ as needed for Proposition~\ref{P:diffeo}.  In this case, both sides of \eqref{GgG identity} are analytic functions of $q$.  Consequently, it suffices to prove the result under the additional hypotheses that $q$ is Schwartz and $\|q\|_{L^\infty}<1$.

Techniques in Sturm--Liouville theory (cf. \cite[\S3.8]{MR0069338}) show that there are solutions $\psi_\pm(x)$ to
\begin{align}\label{ODE}
-\psi'' + q \psi = -\kappa^2 \psi
\end{align}
that decay (along with derivatives) exponentially as $x\to\pm\infty$ and grow exponentially as $x\to\mp\infty$.  Constancy of the Wronskian guarantees that these Weyl solutions (as they are known) are unique up to scalar multiples; we (partially) normalize them by requiring the Wronskian relation
\begin{equation}\label{E:Wron}
\psi_+(x) \psi_-'(x) - \psi_+'(x)\psi_-(x) = 1
\end{equation}
and that $\psi_\pm(x) >0$.  Note that the Sturm oscillation theorem guarantees that neither solution may change sign.

Using the Weyl solutions, we may write the Green's function as
\begin{align}\label{G from psi}
G(x,y) = \psi_+(x\vee y) \psi_-(x\wedge y).
\end{align}
In this way, the proof of the lemma reduces to showing that
\begin{align}\label{pre FTC}
\tfrac12 \int_{-\infty}^x \Bigl[\tfrac{\psi_+(x)}{\psi_+(y)}\Bigr]^2 \,dy + \tfrac12 \int_x^\infty \Bigl[\tfrac{\psi_-(x)}{\psi_-(y)}\Bigr]^2 \,dy = \psi_+(x)\psi_-(x).
\end{align}
However, by \eqref{E:Wron}, we have
\begin{align*}
\tfrac{d\ }{dy} \tfrac{\psi_-(y)}{\psi_+(y)} = \tfrac{1}{\psi_+(y)^2} \qtq{and} \tfrac{d\ }{dy} \tfrac{\psi_+(y)}{\psi_-(y)} = - \tfrac{1}{\psi_-(y)^2}.
\end{align*}
Thus \eqref{pre FTC} follows by the fundamental theorem of calculus and the exponential behavior of $\psi_\pm(y)$, as $|y|\to \infty$.
\end{proof}

\begin{remark}
As mentioned above, there is an alternate integral representation of $\alpha(\kappa;q)$ introduced much earlier.  The proof of Lemma~\ref{L:D 1/g} provides the requisite vocabulary to explain what that is:
\begin{align}\label{E:a from Weyl}
\log[a(i\kappa)] = - \int \tfrac{\psi_+'(y)}{\psi_+(y)} + \kappa \,dy = \int \tfrac{\psi_-'(y)}{\psi_-(y)} - \kappa \,dy.
\end{align}
Here $\psi_\pm$ represent the Weyl solutions; however, the formula applies equally well using the Jost solutions, since they differ only in normalization.  It is in this equivalent form that the first identity appears in \cite[\S3]{MR0303132}.
Averaging these two representations and invoking \eqref{E:Wron} and then \eqref{G from psi} yields
\begin{align}\label{E:a from little g}
\log[a(i\kappa)] = \int \tfrac{1}{2\psi_-(y)\psi_+(y)} - \kappa \,dy = \int \tfrac{1}{2g(y)} - \kappa\,dy,
\end{align}
which is readily seen to be equivalent to \eqref{E:alpha defn}.  One easy way to distinguish these three representations is the fact that $\psi_+(y)$ depends only on the values of $q$ on the interval $[y,\infty)$, while $\psi_-(y)$ is determined by $q$ on the interval $(-\infty,y]$; on the other hand, $g(y)$ depends on the values of $q$ throughout the real line.
\end{remark}

The following identity will be used not only in the proof of Proposition~\ref{P:Intro rho}, but also in Section~\ref{S:3}.

\begin{lemma}\label{L:G ibp}  Given Schwartz functions $f$ and $q$,
\begin{align*}
& \int G(x,y;\kappa,q) \bigl[ - f'''(y) + 2q(y)f'(y) + 2\bigl(q(y)f(y)\bigr)'+4\kappa^2 f'(y)\bigr]G(y,x;\kappa,q)\,dy \\
&= 2 f'(x)g(x;\kappa,q) - 2f(x)g'(x;\kappa,q).
\end{align*}
This identity also holds if merely $f(x)-c$ is Schwartz for some constant $c$.
\end{lemma}

\begin{proof}
The argument that follows applies equally well irrespective of the presence/absence of the constant $c$.  Alternately, as both sides of the identity are linear in $f$, the cases $f$ Schwartz and $f$ constant can be treated separately. However, when $f$ is constant the identity can be obtained more swiftly by other means; see \eqref{translation identity'}.

The most elementary proof proceeds from the defining property of $G$, namely,
$$
\bigl(-\partial_y^2 + q(y) + \kappa^2\bigr)G(y,x) = \bigl(-\partial_y^2 + q(y) + \kappa^2\bigr)G(x,y) = \delta(x-y)
$$
and integration by parts.  However, we find the argument more palatable when presented in terms of operator identities.  Specifically, from the operator identity
\begin{align*}
- f''' &= (-\partial^2+\kappa^2)f' + f' (-\partial^2+\kappa^2) - 2(-\partial^2+\kappa^2)f\partial +  2\partial f(-\partial^2+\kappa^2) - 4\kappa^2f',
\end{align*}
it follows that
\begin{align*}
- R f''' R = f'R - 2 R qf' R + Rf' - 2f\partial R - 2R[\partial,qf]R  + 2 R\partial f - 4\kappa^2 Rf'R.
\end{align*}
Noting, for example, that
$$
g'(x) = \langle\delta_x, [\partial,R]\delta_x\rangle,
$$
the lemma then follows by considering the diagonal of the associated integral kernel.
\end{proof}

\begin{proof}[Proof of Proposition~\ref{P:Intro rho}]
By \eqref{R resolvent},
$$
\tfrac12\int e^{-2\kappa|x-y|} q(y)\,dy = 2\kappa [R_0(2\kappa) q](x).
$$
Combined with Proposition~\ref{P:diffeo}, this shows $\rho\in H^1(\R)$.  Next we write
$$
\rho(x) = 2\kappa^2 \bigl[g - \tfrac1{2\kappa} + \tfrac1{\kappa} R_0(2\kappa) q\bigr](x) - \tfrac{2\kappa^2}{g(x)}[g(x)-\tfrac1{2\kappa}]^2 .
$$
The second summand belongs to $L^1(\R)$ by Proposition~\ref{P:diffeo}; thus it remains to consider the first summand.  To this end, we use \eqref{E:g series} and \eqref{R I2} to obtain
\begin{align}
\int \bigl[g - \tfrac1{2\kappa} + \tfrac1{\kappa} R_0(2\kappa) q\bigr](x) f(x)\,dx &= \sum_{\ell=2}^\infty (-1)^\ell \tr\Bigl\{\sqrt{R_0} f \sqrt{R_0} \Bigl( \sqrt{R_0}\,q\,\sqrt{R_0}\Bigr)^\ell \Bigr\} \notag\\
&\leq \| f \|_{L^\infty}  \Bigl\|\sqrt{R_0}\Big\|_{op}^2 \Bigl\| \sqrt{R_0}\,q\,\sqrt{R_0}\Bigr\|^2_{\I_2} \sum_{\ell=2}^\infty \delta^{\ell-2},\label{E:L1 est}
\end{align}
from which we may conclude that $\rho\in L^1(\R)$.  Note that the arguments just presented actually show that $q\mapsto\rho$ is real analytic as a mapping of $B_\delta$ into $L^1\cap H^1$.

To show convexity at fixed $x$, we compute derivatives.  As in \eqref{O35pp}, the resolvent identity guarantees that
\begin{align}\label{drho}
d[\rho(x)]\bigr|_q (f) =  \tfrac{-1}{2g(x)^2} \int G(x,y) f(y) G(y,x) \,dy + \tfrac12 [e^{-2\kappa|\cdot|} * f ](x)
\end{align}
and thence
\begin{align}\label{ddrho}
d^2[\rho(x)]\bigr|_q (f,h) = {}&{}\tfrac{-1}{g(x)^3} \iint G(x,y) f(y) G(y,x) G(x,z) h(z) G(z,x) \,dy\,dz \\
& + \tfrac{1}{g(x)^2} \iint G(x,y) f(y) G(y,z) h(z) G(z,x) \,dy\,dz. \notag
\end{align}
Multiplying through by $g(x)^3>0$ we then see that the convexity of $\rho(x)$ is reduced to the assertion that
$$
\bigl\langle \sqrt{R} \delta_x,\sqrt{R} \delta_x\bigr\rangle \bigl\langle \sqrt{R} \delta_x,\sqrt{R}fRf\sqrt{R}\,\sqrt{R} \delta_x\bigr\rangle
    - \bigl\langle \sqrt{R} \delta_x,\sqrt{R}f\sqrt{R}\, \sqrt{R} \delta_x\bigr\rangle^2 \geq 0
$$
for all $f\in H^{-1}(\R)$.  (Here inner-products are taken in $L^2(\R)$ which contains $\sqrt{R} \delta_x$.)  The veracity of this assertion now follows immediately from the Cauchy--Schwarz inequality.

Specializing \eqref{drho} to $q\equiv 0$ and substituting \eqref{R resolvent} shows
\begin{align}\label{d rho 0}
\frac{\delta\rho(x)}{\delta q}\biggr|_{q\equiv0} = 0.
\end{align}
Note also that $\rho(x)\equiv0$ when $q\equiv 0$.  In this way the convexity of $q\mapsto \rho(x)$ guarantees its positivity.

Let us now turn our attention to $\alpha(\kappa;q)$.  In view of the preceding, we already know that this is a non-negative, convex, and real-analytic function of $q\in B_\delta$.  It remains to show strict convexity, \eqref{alpha as I2}, and \eqref{O37}.

As we have already noted, $\rho(x)\equiv 0$ when $q\equiv 0$.  Thus \eqref{O37} holds trivially in this case.  In general \eqref{O37} follows easily from
\begin{align}\label{delta alpha}
\frac{\delta \alpha}{\delta q} = \tfrac{1}{2\kappa} - g(x)  = \frac{\delta\ }{\delta q} - \log\det_2\left( 1+ \sqrt{R_0}\, q \, \sqrt{R_0} \right),
\end{align}
which we will now verify.

From \eqref{drho} and Lemma~\ref{L:D 1/g},
\begin{align*}
\frac{d\ }{ds}\biggr|_{s=0} \alpha(\kappa; q+sf) &= - \iint \frac{G(y,x)G(x,y)}{2g(x)^2} f(y)\,dx  \,dy+ \tfrac1{2\kappa}\int f(y)\,dy \\
&= \int \Bigl[\tfrac{1}{2\kappa} - g(x)\Bigr] f(x)\,dx,
\end{align*}
at least for Schwartz functions $f$.  This proves the first equality in \eqref{delta alpha}.

From \eqref{det series} and \eqref{E:g series}, we have
\begin{align*}
\frac{d\ }{ds}\biggr|_{s=0} - \log&\det_2\left( 1+ \sqrt{R_0}\, (q+sf) \, \sqrt{R_0} \right) \\
&= \sum_{\ell=2}^\infty (-1)^\ell\tr\Bigl\{ \Bigl(\sqrt{R_0}\, q \, \sqrt{R_0} \Bigr)^{\ell-1} \sqrt{R_0}\, f \, \sqrt{R_0} \Bigr\} \\
&= \int \Bigl[\tfrac{1}{2\kappa} - g(x)\Bigr] f(x)\,dx.
\end{align*}
This verifies the second equality in \eqref{delta alpha} and so finishes the proof of \eqref{O37}.

Toward verifying strict convexity and \eqref{alpha as I2}, let us first compute the Hessian of $\alpha(\kappa)$ at $q\equiv 0$.  From \eqref{ddrho} and \eqref{R resolvent}, we have
\begin{align}
d^2\alpha\bigr|_{q\equiv 0} (f,f) &= -\tfrac{1}{2\kappa} \iiint e^{-2\kappa|x-y| - 2\kappa|x-z|}  f(y)f(z) \,dx\,dy\,dz \notag\\
&\quad  + \tfrac{1}{2\kappa} \iiint e^{-\kappa|x-y| - \kappa|y-z| - \kappa|z-x|} f(y) f(z) \,dx\,dy\,dz  \label{delta2alpha}\\
&= \tfrac{1}{4\kappa^2} \iint e^{-2\kappa|y-z|} f(y) f(z) \,dy\,dz = \tfrac{1}{\kappa} \int \frac{|\hat f(\xi)|^2}{\xi^2+4\kappa^2}\,d\xi. \notag
\end{align}
As $\alpha(\kappa)$ is real analytic, this immediately shows strict convexity and \eqref{alpha as I2} in some neighbourhood of $q\equiv 0$; however, to verify that the size $\delta$ of this neighbourhood may
be taken independent of $\kappa$, we must adequately control the modulus of continuity of the Hessian.  From \eqref{inverse input} and the first identity in \eqref{delta alpha}, we have
$$
\Bigl|\Bigl( d^2\alpha\bigr|_{q\equiv 0} - d^2\alpha\bigr|_{q}\Bigr)(f,f)\Bigr| \lesssim \delta\kappa^{-1} \int \frac{|\hat f(\xi)|^2}{\xi^2+4\kappa^2}\,d\xi,
$$
thereby settling the matter.
\end{proof}

\section{Dynamics}\label{S:3}

The natural Poisson structure on $\mathcal S(\R)$ or $C^\infty(\R/\Z)$ associated to the KdV equation is
\begin{align}\label{3.0}
\{ F, G \} = \int  \frac{\delta F}{\delta q}(x) \biggl(\frac{\delta G}{\delta q}\biggr) '(x) \,dx .
\end{align}
This structure is degenerate: $q\mapsto \int q$ is a Casimir (i.e. Poisson commutes with everything).  It is common practice to say that this is the
Poisson bracket associated to the (degenerate) almost complex structure $J=\partial_x$ and the $L^2$ inner product.  We shall not need such notions; however, they do suggest a very convenient notation for the time-$t$ flow under the Hamiltonian $H$:
$$
q(t) = e^{t J\nabla\! H} q(0).
$$
Note that under our sign conventions,
$$
\frac{d\ }{dt}\ F\circ e^{t J\nabla\! H} = \{ F, H \} \circ e^{t J\nabla\! H}.
$$

As two simple examples, we note that for
$$
P:=\int \tfrac12 |q(x)|^2\,dx \qtq{and} H_\text{KdV} := \int \tfrac12 |q'(x)|^2 + q(x)^3 \,dx,
$$
we have
\begin{align}\label{trivial delta}
\frac{\delta P}{\delta q}(x) = q(x)  \qtq{and} \frac{\delta H_\text{KdV}}{\delta q}(x) = -q''(x) + 3q(x)^2.
\end{align}
Thus, the flow associated to $P$ is precisely $\partial_t q = \partial_x q$, which is to say, $P$ represents momentum (= generator of translations); the flow associated to $H_\text{KdV}$ is precisely the KdV equation.
Note that $H_\text{KdV}$ and $P$ Poisson commute:
$$
\{H_\text{KdV},P\}=\int \bigl(-q''(x)+3q(x)^2\bigr) q'(x)\,dx = \int \bigl(-\tfrac12 q'(x)^2 + q(x)^3\bigr)' \,dx =0.
$$
This simultaneously expresses that the KdV flow conserves $P$ and that $H_\text{KdV}$ is conserved under translations.  Moreover, the two flows commute:
$$
e^{s J\nabla\! P} \circ e^{t J\nabla\! H_\text{KdV}} = e^{t J\nabla\! H_\text{KdV}} \circ e^{s J\nabla\! P} \qtq{for all} s,t\in\R,
$$
at least as mappings of Schwartz space.  The claim that the KdV flow commutes with translations is without controversy; nonetheless, it is important for what follows to see that it stems precisely from the vanishing of the Poisson bracket.  Fortunately, by restricting our attention to Schwartz-space solutions, we may simply apply the standard arguments from differential geometry; see, for example, \cite[\S39]{MR0997295}.

We will also consider one more Hamiltonian, namely,
\begin{align}\label{H kappa defn}
H_\kappa := - 16 \kappa^5 \alpha(\kappa) + 2 \kappa^2 \int q(x)^2\,dx
\end{align}
which, formally at least, converges to
$$
H_\text{KdV} := \int \tfrac12 |q'(x)|^2 + q(x)^3 \,dx
$$
as $\kappa\to\infty$.  In due course, we will see that $H_\kappa$ leads to a well-posed flow on $H^{-1}$ and that it Poisson commutes with both $P$ and $H_\text{KdV}$, at least as a functional on Schwartz space.  For the moment, however, let us describe the evolution of the diagonal Green's function under the KdV flow.

\begin{prop}\label{L:5.2}
Given $\delta>0$, there is a $\delta_0>0$ so that for every Schwartz solution $q(t)$ to KdV with initial data
$q(0)\in B_{\delta_0}$, we have
\begin{align}\label{prop small}
\sup_{t\in\R}\| q(t)\|_{H^{-1}(\R)} \leq \delta.
\end{align}
Moreover, for each $\kappa\geq 1$, the quantities $g(t,x)=g(x;\kappa,q(t))$, $\rho(t,x)=\rho(x;\kappa,q(t))$, and $\alpha(\kappa;q(t))$ obey
\begin{gather}
\tfrac{d\ }{dt}\, g(t,x)  = -2 q'(t,x) g(t,x) + 2 q(t,x) g'\!(t,x) - 4\kappa^2 g'\!(t,x) \label{E:l5.1c}\\
\tfrac{d\ }{dt}  \, \tfrac{1}{2g(t,x)} = \Bigl( \tfrac{q(t,x)}{g(t,x)} - \tfrac{2\kappa^2}{g(t,x)} + 4\kappa^3\Bigr)' \label{E:l5.1g} \\
\tfrac{d\ }{dt}  \rho(t,x) = \Bigl(\tfrac32 \bigl[e^{-2\kappa|\cdot|}* q^2\bigr](t,x)  + 2q(t,x)\bigl[ \kappa - \tfrac{1}{2g(t,x)}\bigr] - 4\kappa^2\rho(t,x) \Bigr)' \label{E:l5.1h} \\
\tfrac{d\ }{dt} \alpha(\kappa;q(t)) = 0. \label{E:l5.1z}
\end{gather}
\end{prop}

\begin{proof}
Without loss of generality, we may require that $\delta$ is a small as we wish.  We shall require that $\delta$ meets the requirements of Propositions~\ref{P:diffeo},~\ref{P:elliptic}, and~\ref{P:Intro rho}.  As an initial choice, we then set $\delta_0=\tfrac12\delta$.  This guarantees that these propositions are all applicable to $q(t)$ for some open interval of times containing $t=0$.  (Schwartz solutions are necessarily continuous in $H^{-1}(\R)$.)  We will show below that equations \eqref{E:l5.1c}--\eqref{E:l5.1z} are valid on this time interval. But then, choosing $\kappa=1$ in \eqref{alpha as I2} and \eqref{E:l5.1z}, we obtain
$$
\| q(t) \|_{H^{-1}} \lesssim \| q(0) \|_{H^{-1}}
$$
on this interval.  Thus we see that \eqref{prop small} holds globally in time, after updating our choice of $\delta_0$, if necessary.

It remains to show that the stated differential equations apply to Schwartz solutions whose $H^{-1}$ norm is small enough that the results of Section~\ref{S:2} apply.  We begin the proof in earnest after one minor preliminary: by taking an $h$ derivative in \eqref{translation identity} and using the resolvent identity, we have
\begin{align}\label{translation identity'}
g'(x; q) =  - \int G(x,y) q'(y) G(y,x)\,dy.
\end{align}

By the resolvent identity, then Lemma~\ref{L:G ibp}, and then \eqref{translation identity'},
\begin{align*}
\frac{d\ }{dt} g(&x; q(t)) =  - \int G(x,y) \bigl[ -q'''(t,y) + 6q(t,y)q'(t,y) \bigr]  G(y,x)\,dy \\
&= - 2 q'(t,x)g(x; q(t)) + 2 q(t,x) g'(x; q(t)) + 4\kappa^2 \int G(x,y) q'(t,y)  G(y,x)\,dy \\
&= - 2 q'(t,x)g(x; q(t)) + 2 q(t,x) g'(x; q(t)) - 4\kappa^2 g'(x;q(t)).
\end{align*}
This proves \eqref{E:l5.1c}.  Alternately, \eqref{E:l5.1c} can be derived from the Lax pair formulation of KdV; specifically,
\begin{align*}
\frac{d\ }{dt} \bigl(L(t)+\kappa^2\bigr)^{-1} = \bigl[ P(t), \bigl(L(t)+\kappa^2\bigr)^{-1} \bigr] .
\end{align*}
We leave the details to the interested reader.

Equation \eqref{E:l5.1g} follows immediately from \eqref{E:l5.1c} and the chain rule, while \eqref{E:l5.1h} is simply a combination of \eqref{E:l5.1g} and \eqref{KdV}.  Lastly, \eqref{E:l5.1z} follows from integrating \eqref{E:l5.1h} in $x$ over the whole line.
\end{proof}

\begin{remark} Combining \eqref{E:l5.1a} and \eqref{E:l5.1c} yields
\begin{align}
\tfrac{d\ }{dt}\, g(x) &= \Bigl( 2g''(x) -6q(x)g(x) - 12\kappa^2g(x) + 6\kappa \Bigr)' , \label{E:l5.1c'}
\end{align}
from which we see that there is also a microscopic conservation law for the KdV flow associated to $g(x)$. Ultimately, however, this turns out to be a consequence of the conservation of $\alpha(\kappa)$; specifically, we have
$$
\frac{d\ }{d\kappa} \alpha(\kappa) = - 2\kappa  \int g(x) - \tfrac{1}{2\kappa} + \tfrac{1}{4\kappa^3} q(x) \,dx.
$$
\end{remark}

\begin{prop}\label{P:H kappa}
Fix $\kappa\geq 1$. The Hamiltonian evolution induced by $H_\kappa$ is
\begin{align}\label{H kappa flow q}
\tfrac{d\ }{dt} q(x) = 16\kappa^5 g'(x;\kappa) + 4\kappa^2 q'(x).
\end{align}
This flow is globally well-posed for initial data in $B_\delta$, for $\delta>0$ small enough (independent of $\kappa$), and conserves $\alpha(\varkappa)$ for any $\varkappa\geq 1$.  Moreover, in the case of Schwartz-class initial data, the solution is Schwartz-class for all time, the associated diagonal Green's function evolves according to
\begin{align}\label{H kappa flow g}
\tfrac{d\ }{dt}  \, \tfrac{1}{2g(x;\vk)} &= - \tfrac{4\kappa^5}{\kappa^2-\vk^2} \Bigl( \tfrac{g(x;\kappa)}{g(x;\vk)} -\tfrac{\vk}{\kappa} \Bigr)' + 4\kappa^2 \Bigl( \tfrac{1}{2g(x;\vk)} - \vk \Bigr)'
    \quad\text{if $\vk\neq \kappa$},
\end{align}
and the flow commutes with that of $H_\text{KdV}$.
\end{prop}

\begin{proof}
From \eqref{delta alpha} and \eqref{trivial delta} we see that
$$
\frac{\delta H_\kappa}{\delta q} = - 16 \kappa^5\bigl[\tfrac{1}{2\kappa} - g(x;\kappa,q)\bigr]  + 4 \kappa^2 q(x) ,
$$
from which \eqref{H kappa flow q} immediately follows.

Rewriting \eqref{H kappa flow q} as the integral equation
$$
q(t,x) = q(0,x+4\kappa^2 t) + \int_0^t 16\kappa^5 g'\bigl(x+4\kappa^2(t-s);\kappa,q(s)\bigr) \,ds ,
$$
we see that local well-posedness follows by Picard iteration and the estimate
$$
\bigl\| g'(x,q) - g'(x,\tilde q) \bigr\|_{H^{-1}} \lesssim \bigl\| g(x,q) - g(x,\tilde q) \bigr\|_{H^{1}} \lesssim \| q -\tilde q \|_{H^{-1}},
$$
which in turn follows from the diffeomorphism property.

Global well-posedness follows from local well-posedness, once we prove that $\alpha(\varkappa)$ is conserved, since we may then use \eqref{alpha as I2} to guarantee that the solution remains small in $H^{-1}$.  (This argument appeared already in Proposition~\ref{L:5.2}.)  Moreover, because the problem is $H^{-1}$-locally well-posed, it suffices to verify conservation of $\alpha(\vk)$ just in the case of Schwartz initial data.  Note that \eqref{g stronger mapping} shows that solutions with Schwartz initial data remain in Schwartz class.  So let us consider a Schwartz solution $q(t)$ to \eqref{H kappa flow q} and endeavor to prove conservation of $\alpha(\varkappa)$.  Actually, it suffices to prove \eqref{H kappa flow g}, because conservation of $\alpha(\varkappa)$ follows from this and \eqref{H kappa flow q}.

By the resolvent identity and \eqref{H kappa flow q},
\begin{align*}
\tfrac{d\ }{dt} \tfrac{1}{2 g(t,x;\vk)}  ={} & \tfrac{8\kappa^5}{g(t,x;\vk)^2} \int G(x,y;\vk,q(t)) g'(t,y;\kappa) G(y,x;\vk,q(t))\,dy \\
& + \tfrac{2\kappa^2}{g(t,x;\vk)^2} \int G(x,y;\vk,q(t)) q'(t,y) G(y,x;\vk,q(t))\,dy.
\end{align*}
From here we substitute the following rewriting of \eqref{E:l5.1a}
$$
4(\kappa^2-\vk^2) g'(y;\kappa) = -\bigl[ - g'''(y;\kappa) + 2\bigl(q(y)g(y;\kappa)\bigr)' + 2 q(y)g'(y;\kappa) + 4\vk^2g'(y;\kappa)\bigr]
$$
into the first term and use Lemma~\ref{L:G ibp}, while for the second term we employ \eqref{translation identity'}.  In this way, we deduce that
\begin{align*}
\tfrac{d\ }{dt} \tfrac{1}{2 g(t,x;\vk)}  &= - \tfrac{4\kappa^5}{(\kappa^2-\vk^2) g(t,x;\vk)^2} \bigl[ g'(t,x;\kappa)g(t,x;\vk)  - g(t,x;\kappa)g'(t,x;\vk) \bigr] \\
&\qquad - \tfrac{2\kappa^2}{g(t,x;\vk)^2} g'(t,x;\vk),
\end{align*}
which agrees with \eqref{H kappa flow g}.

Lastly, by \eqref{H kappa defn} and Proposition~\ref{L:5.2},
\begin{align*}
\{ H_\kappa, H_\text{KdV} \} = - 16 \kappa^5 \{ \alpha(\kappa), H_\text{KdV} \} + 4 \kappa^2 \{ P, H_\text{KdV} \} = 0,
\end{align*}
which shows that the $H_\kappa$ and $H_\text{KdV}$ flows commute, at least as mappings on Schwartz space.
\end{proof}

\section{Equicontinuity}\label{S:4}

Let us first recall the meaning of equicontinuity:

\begin{definition}
A subset $Q$ of $H^s$ is said to be \emph{equicontinuous} if
\begin{gather}
q(x+h) \to q(x) \quad\text{in $H^s$ as $h\to 0$, uniformly for $q\in Q$.} \label{E:equi1}
\end{gather}
\end{definition}

This definition works in great generality.  For $H^s$ spaces, it is also common to define equicontinuity as tightness of the Fourier transform.  The two approaches are easily reconciled, as our next lemma shows.

\begin{lemma}\label{L:equi 1}
Fix $-\infty < \sigma < s <\infty$. Then:\\
(i) A bounded subset $Q$ of $H^s(\R)$ is equicontinuous in $H^s(\R)$ if and only if
\begin{gather}
\int_{|\xi|\geq \kappa} |\hat q(\xi)|^2 (\xi^2+4)^s \,d\xi \to 0  \qtq{as $\kappa\to \infty$, uniformly for $q\in Q$.} \label{E:equi2}
\end{gather}
(ii) A sequence $q_n$ is convergent in $H^s(\R)$ if and only if it is convergent in $H^\sigma(\R)$ and equicontinuous in $H^s(\R)$.

\end{lemma}

\begin{proof}
As $Q$ is bounded and
\begin{align*}
\int |e^{i\xi h}-1|^2 |\hat q(\xi)|^2 (\xi^2+4)^s \,d\xi &\lesssim \kappa^2 h^2 \int |\hat q(\xi)|^2 (\xi^2+4)^s \,d\xi \\
&\qquad  + \int_{|\xi|>\kappa} |\hat q(\xi)|^2 (\xi^2+4)^s \,d\xi,
\end{align*}
we see that \eqref{E:equi2} implies \eqref{E:equi1}.  To prove the converse, we note that
\begin{align*}
\int |e^{i\xi h}-1|^2\, \kappa e^{-2\kappa|h|}\,dh &= \tfrac{2\xi^2}{\xi^2+4\kappa^2} \gtrsim 1 - \chi_{[-\kappa,\kappa]}(\xi)
\end{align*}
and hence
\begin{align*}
\int\, \| q(x+h) - q(x) \|_{H^s(\R)}^2 \kappa e^{-2\kappa|h|}\,dh \gtrsim \int_{|\xi|>\kappa} |\hat q(\xi)|^2 (\xi^2+4)^s \,d\xi.
\end{align*}

Let us now turn attention to (ii).  As the forward implication is trivial, we need only consider sequences $q_n$ that are convergent in $H^\sigma(\R)$ and equicontinuous in $H^s(\R)$.  But then writing
\begin{align*}
\int |\hat q_n(\xi) - \hat q_m(\xi)|^2 (\xi^2+4)^s \,d\xi &\leq (\kappa^2+4)^{s-\sigma} \int |\hat q_n(\xi) - \hat q_m(\xi)|^2 (\xi^2+4)^\sigma \,d\xi \\
&\qquad + \int_{|\xi|>\kappa} |\hat q_n(\xi) - \hat q_m(\xi)|^2 (\xi^2+4)^s \,d\xi
\end{align*}
and employing \eqref{E:equi2}, we see that the sequence is Cauchy in $H^{s}(\R)$ and so convergent there also.
\end{proof}

It is now easy to see that equicontinuity in $H^{-1}(\R)$ is readily accessible through the conserved quantity $\alpha(\kappa;q)$:

\begin{lemma}\label{L:equi 2}
A subset $Q$ of $B_\delta$ is equicontinuous in $H^{-1}(\R)$ if and only if
\begin{gather}
\kappa \alpha(\kappa;q) \to 0  \quad\text{as $\kappa\to \infty$, uniformly for $q\in Q$.} \label{E:equi3}
\end{gather}
\end{lemma}

\begin{proof}
By virtue of \eqref{alpha as I2}, it suffices to show that $Q$ is equicontinuous in $H^{-1}(\R)$ if and only if
\begin{gather}
\lim_{\kappa\to\infty} \ \sup_{q\in Q}\ \int_{\R} \frac{|\hat q(\xi)|^2}{\xi^2+4\kappa^2} \,d\xi = 0. \label{E:equi3'}
\end{gather}

That \eqref{E:equi3'} implies \eqref{E:equi2} and hence equicontinuity follows immediately from
\begin{align*}
\int_{|\xi|\geq \kappa} \frac{|\hat q(\xi)|^2}{\xi^2+4} \,d\xi \lesssim \int_{\R} \frac{|\hat q(\xi)|^2}{\xi^2+4\kappa^2} \,d\xi .
\end{align*}
On the other hand, \eqref{E:equi2} implies \eqref{E:equi3'} by virtue of the boundedness of $Q$ and
\begin{align*}
\int \frac{|\hat q(\xi)|^2}{\xi^2+4\kappa^2} \,d\xi &\lesssim \tfrac{\vk^2}{\kappa^2} \int \frac{|\hat q(\xi)|^2}{\xi^2+4} \,d\xi + \int_{|\xi|>\vk} \frac{|\hat q(\xi)|^2 \,d\xi}{\xi^2+4}.
\qedhere
\end{align*}
\end{proof}

From the preceding lemma and the conservation of $\alpha(\kappa)$ we readily deduce the following:

\begin{prop}\label{P:equi}
Let $Q\subset B_\delta$ be a set of Schwartz functions that is equicontinuous in $H^{-1}(\R)$.  Then
\begin{align}\label{Q star}
Q^* = \bigl\{ e^{J\nabla(t H_\text{KdV} + s H_\kappa)} q : q\in Q,\ t,s\in \R,\text{ and } \kappa\geq 1 \bigr\}
\end{align}
is equicontinuous in $H^{-1}(\R)$.  By virtue of this,
\begin{align}\label{uniform to q}
4\kappa^3\bigl[ \tfrac1{2\kappa} - g(x;\kappa,q) \bigr] \to q \quad\text{in $H^{-1}(\R)$ as $\kappa\to\infty$},
\end{align}
uniformly for $q\in Q^*$.
\end{prop}

\begin{proof}
By Lemma~\ref{L:equi 2} and \eqref{alpha as I2}, the boundedness and equicontinuity of $Q$ guarantees that $\alpha(\kappa;q)$ is uniformly bounded on $Q$ and that
$$
\lim_{\kappa\to\infty} \kappa \alpha(\kappa;q) = 0 \quad\text{uniformly for $q\in Q$.}
$$
But then since $\alpha(\kappa;q)$ is conserved under these flows, we may reverse this reasoning to deduce that $Q^*$ is equicontinuous as well.

Looking back to \eqref{E:L1 est}, \eqref{R I2}, and \eqref{alpha as I2}, we have
$$
\kappa^3 \bigl\| \tfrac1{2\kappa}  - g(\,\cdot\;;\kappa,q) - \tfrac1{\kappa} R_0(2\kappa) q\bigr\|_{L^1} \lesssim \kappa \alpha(\kappa;q),
$$
which converges to zero as $\kappa\to\infty$ uniformly for $q\in Q^*$ by \eqref{E:equi3}.  In this way, the proof of \eqref{uniform to q} is reduced to the simple calculation
$$
\| 4\kappa^2 R_0(2\kappa)q - q \|_{H^{-1}}^2 = \int \frac{\xi^4 |\hat q(\xi)|^2}{(\xi^2+4\kappa^2)^2}\,\frac{d\xi}{\xi^2+4} \leq \int \frac{|\hat q(\xi)|^2}{\xi^2+4\kappa^2}\,d\xi
$$
and \eqref{E:equi3'}.
\end{proof}

\section{Well-posedness}\label{S:5}

\begin{theorem}\label{T:converge}
Let $q_n(t)$ be a sequence of Schwartz solutions to \eqref{KdV} on the line and fix $T>0$.  If $q_n(0)$ converges in $H^{-1}(\R)$ then so does $q_n(t)$, uniformly for $t\in[-T,T]$.
\end{theorem}

\begin{proof}
Let us first reduce to the case $q_n(0)\in B_\delta$ for any fixed $\delta>0$, which is required in order to apply many of the results of the previous sections.  This is easily handled by a simple scaling argument:
if $q(t,x)$ is a Schwartz solution to \eqref{KdV}, then so is
\begin{align}\label{q scaling}
q_\lambda(t,x) = \lambda^2 q(\lambda^3 t, \lambda x)
\end{align}
for any $\lambda>0$; moreover,
\begin{align}\label{H-1 scaling}
\| q_\lambda(0) \|_{H^{-1}(\R)}^2 = \lambda \int \frac{|\hat q(0,\xi)|^2\,d\xi}{\xi^2+4\lambda^{-2}},
\end{align}
which converges to zero as $\lambda\to 0$.  Although it is incidental to the current proof, let us note here that
\begin{equation}\label{G scaling}
\begin{gathered}
G(x,y;\kappa,q_\lambda)  = \lambda^{-1} G(\lambda x,\lambda y;\lambda^{-1}\kappa,q),\\
\rho(x;\kappa,q_\lambda) =\lambda\rho(\lambda x;\lambda^{-1}\kappa,q),\qtq{and} \alpha(\kappa;q_\lambda) = \alpha(\lambda^{-1}\kappa;q).
\end{gathered}
\end{equation}

By commutativity of the flows, we have
\begin{align*}
q_n(t) = e^{t J\nabla(H_\text{KdV} - H_\kappa)} \circ e^{t J\nabla H_\kappa} q_n(0).
\end{align*}
Thus, setting $Q=\{q_n(0)\}$ and defining $Q^*$ as in \eqref{Q star}, we have
\begin{align}
\sup_{|t|\leq T} \| q_n(t) - q_m(t) \|_{H^{-1}}  &\leq \sup_{|t|\leq T} \| e^{t J\nabla H_\kappa} q_n(0) - e^{t J\nabla H_\kappa} q_m(0) \|_{H^{-1}} \label{diff 1}\\
&\qquad + 2 \sup_{q\in Q^*} \sup_{|t|\leq T} \| e^{tJ\nabla (H_\text{KdV} - H_\kappa)} q - q \|_{H^{-1}}. \notag
\end{align}
Note that $Q^*$ is equicontinuous in $H^{-1}(\R)$; this follows from Proposition~\ref{P:equi}.

For fixed $\kappa$, the first term in RHS\eqref{diff 1} converges to zero as $n,m\to\infty$ due to the well-posedness of the $H_\kappa$ flow; see Proposition~\ref{P:H kappa}.  Thus, it remains to prove that
\begin{align}\label{56}
\lim_{\kappa\to\infty} \sup_{q\in Q^*} \ \sup_{|t|\leq T}\ \| e^{tJ\nabla (H_\text{KdV} - H_\kappa)} q - q \|_{H^{-1}} =0.
\end{align}

We prove \eqref{56} by considering the reciprocal of the diagonal Green's function at some fixed energy.  To this end, we fix $\varkappa\geq 1$ and adopt the following notations: given $q\in Q^*$ and $\kappa\geq \vk +1$,
$$
q(t) := e^{tJ\nabla (H_\text{KdV} - H_\kappa)} q \qtq{and} g(t,x;\varkappa) := g(x;\varkappa,q(t)).
$$
Note that $q(t)\in (Q^*)^*=Q^*$ for any $t\in\R$.

Combining \eqref{E:l5.1g} and \eqref{H kappa flow g}, we obtain
\begin{align*}
\tfrac{d\ }{dt}  \tfrac{1}{2g(t,x;\vk)} &= \Bigl\{ \tfrac{1}{g(t,x;\vk)} \Bigl( q(t,x) + \tfrac{4\kappa^5}{\kappa^2-\vk^2}\bigl[g(t,x;\kappa)-\tfrac{1}{2\kappa}\bigr]  -   \tfrac{4\vk^5}{\kappa^2-\vk^2} \bigl[ g(t,x;\vk) -\tfrac{1}{2\vk}\bigr]\Bigr)\Bigr\}'
\end{align*}
and thence
\begin{align*}
 \bigl\| \tfrac{d\ }{dt} \bigl(\varkappa - \tfrac{1}{2g(t;\vk)}\bigr) \bigr\|_{H^{-2}}
&\lesssim \bigl\| q(t,x) + 4\kappa^3\bigl[g(t,x;\kappa)-\tfrac{1}{2\kappa}\bigr]\bigr\|_{H^{-1}} \\
&\quad \ {} + \kappa \bigl\| g(t,x;\kappa)-\tfrac{1}{2\kappa}\bigr\|_{H^{-1}}  + \kappa^{-2} \bigl\| g(t,x;\varkappa)-\tfrac{1}{2\varkappa}\bigr\|_{H^{-1}}
\end{align*}
uniformly for $q\in Q^*$ and $\kappa\geq \varkappa+1$.  (The implicit constants here depend on $\varkappa$.)  But then, by the fundamental theorem of calculus and Proposition~\ref{P:equi},
\begin{align}
\lim_{\kappa\to\infty}\; \sup_{q\in Q^*}\ \sup_{|t|\leq T}\ \bigl\| \tfrac{1}{2g(t;\vk)} - \tfrac{1}{2g(0;\vk)} \bigr\|_{H^{-2}} =0.
\end{align}
In view of Lemma~\ref{L:equi 1}(ii), we may upgrade this convergence to
\begin{align}\label{60}
\lim_{\kappa\to\infty} \; \sup_{q\in Q^*} \ \sup_{|t|\leq T}\ \bigl\| \tfrac{1}{2g(t;\vk)} - \tfrac{1}{2g(0;\vk)} \bigr\|_{H^{1}} =0,
\end{align}
due to the equicontinuity of the set
$$
E:= \Bigl\{ \varkappa - \tfrac{1}{2g(x;\varkappa,q(t))} \in H^1(\R) : q\in Q^* \text{ and } t\in\R\Bigr\}
$$
in $H^1(\R)$.  This property of $E$ holds because, by the diffeomorphism property and the relation \eqref{translation identity}, it is equivalent to equicontinuity of $Q^*$.

Lastly, the diffeomorphism property shows that \eqref{60} implies \eqref{56} and so completes the proof of Theorem~\ref{T:converge}.
\end{proof}

The line case of Theorem~\ref{T:main} follows from the next corollary.  We then extend this to higher values of $s$.

\begin{corollary}\label{C:1}
The KdV equation is globally well-posed in $H^{-1}(\R)$ in the following sense:  The solution map extends (uniquely) from Schwartz space to a jointly continuous map
$$
\Phi:\R\times H^{-1}(\R)\to H^{-1}(\R).
$$
In particular, $\Phi$ has the group property: $\Phi(t+s)=\Phi(t)\circ \Phi(s)$.  Moreover, each orbit $\{\Phi(t,q) : t\in\R\}$ is bounded and equicontinuous in $H^{-1}(\R)$.  Concretely,
\begin{align}\label{global bound}
\sup_t \| q(t) \|_{H^{-1}(\R)} \lesssim \| q(0) \|_{H^{-1}(\R)}  + \| q(0) \|_{H^{-1}(\R)}^3 .
\end{align}
\end{corollary}

\begin{proof}
Given $q\in H^{-1}(\R)$, we may define $\Phi(t,q)$ by choosing some sequence of Schwartz solutions $q_n(t)$ with $q_n(0)\to q$ in $H^{-1}(\R)$ and then set
$$
\Phi(t,q)= \lim_{n\to\infty} q_n(t).
$$
By virtue of Theorem~\ref{T:converge}, this limit exists in $H^{-1}(\R)$,  it is independent of the sequence $q_n$, and the convergence is uniform on compact intervals of time.

Now consider a sequence $q_n\to q \in H^{-1}(\R)$ and fix $T>0$.  Theorem~\ref{T:converge} guarantees that there is a sequence of Schwartz solutions $\tilde q_n$ so that
$$
\sup_{|t|\leq T} \| \tilde q_n(t) - \Phi(t,q_n) \|_{H^{-1}} \to 0 \qtq{as $n\to\infty$.}
$$
But then $\tilde q_n(0) \to q$ in $H^{-1}$ and so Theorem~\ref{T:converge} implies
$$
\sup_{|t|\leq T} \| \tilde q_n(t) - \Phi(t,q) \|_{H^{-1}} \to 0 \qtq{as $n\to\infty$.}
$$
As each $\tilde q_n(t)$ is itself $H^{-1}(\R)$-continuous in time, this proves joint continuity of $\Phi$.

As $\Phi$ is continuous, the group property on $H^{-1}(\R)$ is inherited from that on Schwartz space.

For small initial data, boundedness and equicontinuity of orbits follows from conservation of $\alpha(\kappa)$, \eqref{alpha as I2}, and Lemma~\ref{L:equi 2}.  In fact, this argument shows that
$$
\sup_t \| q(t) \|_{H^{-1}(\R)} \lesssim \| q(0) \|_{H^{-1}(\R)} \qtq{for} q(0)\in B_\delta
$$
and $\delta>0$ sufficiently small.  Equicontinuity and \eqref{global bound} for large data then follow from the scaling transformation \eqref{q scaling}.
\end{proof}

\begin{corollary}\label{C:2}
The KdV equation is globally well-posed in $H^{s}(\R)$ for all $s\geq -1$.
\end{corollary}

\begin{proof}
In view of the preceding, it suffices to prove an analogue of Theorem~\ref{T:converge} in $H^s(\R)$.  We shall content ourselves with the treatment of $s\in(-1,0)$ here, since we may do so in a simple and uniform manner; moreover, together with Corollary~\ref{C:1}, this covers all cases not previously known, namely, $s\in[-1,-\tfrac34)$.

Given Schwartz solutions $q_n(t)$ to \eqref{KdV} with $q_n(0)$ convergent in $H^{s}(\R)$ and $T>0$, we may apply Theorem~\ref{T:converge} to obtain convergence of $q_n(t)$ in $H^{-1}(\R)$, uniformly for $t\in[-T,T]$.  The goal is to upgrade this to uniform convergence in $H^s(\R)$.  In view of Lemma~\ref{L:equi 1}, this amounts to demonstrating $H^s(\R)$-equicontinuity of the set $\{q_n(t):n\in\N\text{ and } t\in[-T,T]\}$.

To prove equicontinuity, we employ the following trick we used in \cite{KVZ}:  Integrating both sides of \eqref{alpha as I2} against the measure $\kappa^{2+2s} \,d\kappa$ over the interval $[\kappa_0,\infty)$, we obtain
\begin{align}\label{int alpha}
\int_{\kappa_0}^\infty \alpha(\kappa;q) \kappa^{2+2s} \,d\kappa \approx \int |\hat q(\xi)|^2 (\xi^2+4\kappa_0^2)^s \,d\xi,
\end{align}
where the implicit constants depend only on $s$.  Notice that LHS\eqref{int alpha} is conserved by the flow and so, it follows that
\begin{align}\label{int alpha;}
\int |\hat q_n(t,\xi)|^2 (\xi^2+4\kappa_0^2)^s \,d\xi \approx \int |\hat q_n(0,\xi)|^2 (\xi^2+4\kappa_0^2)^s \,d\xi
\end{align}
uniformly for $n\in\mathbb{N}$ and $t\in\R$.  As the initial data $q_n(0)$ are $H^s(\R)$-convergent, they are $H^s(\R)$-equicontinuous and so RHS\eqref{int alpha;} converges to zero as $\kappa_0\to\infty$ uniformly in $n$.  But then LHS\eqref{int alpha;} converges to zero as $\kappa_0\to\infty$ uniformly in $n$, thus proving equicontinuity of $\{q_n(t) : n\in\mathbb{N}\text{ and } t\in\R\}$.
\end{proof}

\section{The periodic case}\label{S:periodic}

With the exception of Section~\ref{S:2}, very little of substance changes in carrying over the arguments presented so far to the case of KdV with initial data in $H^{-1}(\R/\Z)$.  Nevertheless, there are several reasons why we chose not to present both geometries simultaneously (as in our prior work \cite{KVZ}).  First and foremost, we avoid the necessity of continually interrupting the principal line of reasoning to discuss minor changes (often notational) associated to the two geometries.

Secondly, in the non-periodic setting, the scaling transformation \eqref{q scaling} allows us effortlessly to focus attention on small solutions, which manifests in the appearance of $\delta$ throughout our arguments thus far.  To overcome the lack of scaling-invariance in the periodic case, we follow the approach we used in \cite{KVZ}.  Although we still maintain that this is the best solution, we can attest that it burdens the exposition considerably.  Looking to \eqref{H-1 scaling} and \eqref{G scaling}, we see that rescaling $q$ transforms the parameter $\kappa$. Correspondingly, the smallness condition for $q$ can be replaced by a relation involving $\kappa$ and $q$.  On the other hand, many formulas become tremendously ugly if we do not employ the simplifications made possible by requiring $\kappa\geq 1$.  This reasoning leads to the following \emph{coupled} conditions on $q$ and $\kappa$ that we shall impose:
\begin{equation}\label{periodic smallness}
\kappa\geq 1 \qtq{and} \kappa^{-1/2} \| q \|_{H^{-1}(\R/\Z)} \leq \delta.
\end{equation}
Here $\delta>0$ remains our over-arching smallness parameter, whose value will be allowed to shrink as the argument progresses.  Concomitant with this, for fixed $\kappa\geq 1$, we define
\begin{equation}\label{periodic B delta}
B_{\delta,\kappa} := \{ q\in H^{-1}(\R/\Z) : \kappa^{-1/2} \| q \|_{H^{-1}(\R/\Z)} \leq \delta\bigr\}.
\end{equation}

To make the arguments as parallel as possible, we shall insist on working with a Lax operator
$$
L = -\partial_x^2 + q(x)
$$
(and its resolvent) acting on $L^2(\R)$ with periodic coefficients and \emph{not} as an operator on $L^2(\R/\Z)$.  This deviates from our treatment in \cite{KVZ}.

In light of our convention, $L$ is no longer a relatively Hilbert--Schmidt (or even relatively compact) perturbation of the case $q\equiv 0$.  As many arguments in Section~\ref{S:2} were founded on \eqref{R I2}, which is the quantitative expression of this, those arguments do not automatically carry over to the periodic case.  In Lemma~\ref{L:A.1} we obtain the key substitute for \eqref{R I2}.  We will then show how to use this to obtain the analogues of the results from Section~\ref{S:2} in the periodic setting.

By comparison, Section~\ref{S:3} is almost devoid of estimates (and those that do appear are easily adapted).  Rather, it is preoccupied with identities that hold pointwise in space and so are immune to the ambient geometry.

Once we have proved \eqref{periodic alpha} below, everything in Section~\ref{S:4} carries over by simply replacing every instance of integration with respect to $\xi$ by summation over $\xi\in2\pi\Z$.

The principal difficulty in transferring the proof of Theorem~\ref{T:converge} to the circle case is the absence of the scaling symmetry \eqref{q scaling}; we have already explained how this can be avoided.  The only change needed for the treatment of the remaining results in Section~\ref{S:5}, namely, Corollaries~\ref{C:1} and~\ref{C:2}, is to employ \eqref{periodic alpha} whenever the original argument calls on \eqref{alpha as I2}.

Let us turn now to the central matter at hand, namely, obtaining analogues of the principal results of Section~\ref{S:2} in the periodic setting.

\begin{lemma}\label{L:A.1}
Fix $\psi\in C^\infty_c(\R)$.  If $q,f\in H^{-1}(\R/\Z)$, then
\begin{align}
\bigl\| \sqrt{R_0}\,q \sqrt{R_0} \bigr\|_{L^2(\R)\to L^2(\R)}^2 &\lesssim \kappa^{-1} \sum_{\xi\in2\pi\Z} \frac{|\hat q(\xi)|^2}{\xi^2+4\kappa^2},\label{A.1.1} \\
\bigl\| \sqrt{R_0}\, f\psi R_0 q \sqrt{R_0} \bigr\|_{\I_1(L^2(\R))} &\lesssim \kappa^{-1} \| f \|_{H^{-1}_\kappa(\R/\Z)} \| q \|_{H^{-1}_\kappa(\R/\Z)},\label{A.1.2}
\end{align}
both uniformly for $\kappa\geq 1$.
\end{lemma}

Note that $\I_1(L^2(\R))$ denotes the ideal of trace-class operators acting on the Hilbert space $L^2(\R)$.  Here and below, we simply use trace-class as a notational convenience for denoting operators representable as a product of Hilbert--Schmidt operators:
\begin{align}\label{I1 from I2}
\| B \|_{\I_1} = \inf \bigl\{ \|B_1\|_{\I_2}\|B_1\|_{\I_2} : B = B_1 B_2 \bigr\} .
\end{align}
For a proper discussion of trace-class, including the veracity of \eqref{I1 from I2}, see \cite{MR2154153}.

Before beginning the proof of Lemma~\ref{L:A.1}, we describe one more preliminary: Given $f\in L^2(\R)$ and $\theta\in[0,2\pi]$, we define
$$
f_\theta(x) = \sum_{\xi\in2\pi\Z} \hat f(\xi+\theta) e^{ix(\xi+\theta)},
$$
which may be regarded as a jointly square-integrable function of $x\in[0,1]$ and $\theta\in[0,2\pi]$.  Indeed,
$$
\int_{\R} |f(x)|^2\,dx = \int_0^{2\pi} \| f_\theta\|_{L^2([0,1])}^2\,d\theta.
$$
Moreover, any (pseudo)differential operator $L$ with $1$-periodic coefficients acts fibre-wise, which is to say it commutes with multiplication by any function of $\theta$.  Note that what we describe here is simply the standard direct integral representation of a periodic operator (cf. \cite[\S XIII.16]{MR0493421}).

\begin{proof}[Proof of Lemma~\ref{L:A.1}]
As the operator appearing in \eqref{A.1.1} is self-adjoint, it suffices to take $f\in L^2(\R)$ and consider
\begin{align*}
\langle f, \sqrt{R_0} q \sqrt{R_0} \,f\rangle_{L^2} = \int_0^{2\pi} \langle f_\theta, \mcM_\theta q \mcM_\theta f_\theta \rangle\,d\theta,
\end{align*}
where $\mcM_\theta :L^2([0,1])\to L^2([0,1])$ is defined via
$$
\mcM_\theta  : \sum_{\xi\in2\pi\Z} c_\xi e^{ix(\xi+\theta)} \mapsto \sum_{\xi\in2\pi\Z} \frac{c_\xi e^{ix(\xi+\theta)}}{\sqrt{(\xi+\theta)^2+\kappa^2}}.
$$
In this way, we see that
$$
\| \sqrt{R_0} q \sqrt{R_0} \|_{L^2(\R)\to L^2(\R)}  = \bigl\| \| \mcM_\theta q \mcM_\theta \|_{L^2([0,1])\to L^2([0,1])} \bigr\|_{L^\infty_\theta}.
$$
The estimate \eqref{A.1.1} now follows by bounding operator norms by Hilbert--Schmidt norms and the equivalence
\begin{align}\label{mcM norm}
\| \mcM_\theta q \mcM_\theta \|_{\I_2(L^2([0,1]))}^2  \approx \kappa^{-1} \sum_{\xi\in2\pi\Z} \frac{|\hat q(\xi)|^2}{\xi^2+4\kappa^2},
\end{align}
which is valid uniformly for $\theta\in[0,2\pi)$.

We turn now to \eqref{A.1.2}. From \eqref{basic commutators}, we find
\begin{align*}
\sqrt{R_0}\, f\psi R_0 q \sqrt{R_0}\, &= \sqrt{R_0}\, f\psi\langle x\rangle R_0 \langle x\rangle^{-1} q \sqrt{R_0}
	+ \sqrt{R_0}\, f\psi \sqrt{R_0}\, A \sqrt{R_0}\, \langle x\rangle^{-1} q \sqrt{R_0} \\
\text{with}\quad A &= \sqrt{R_0}\, \bigl(\tfrac{x}{\langle x\rangle} \partial_x + \partial_x\tfrac{x}{\langle x\rangle}\bigr) \sqrt{R_0}\,.
\end{align*}
Evidently, $A$ is an $L^2(\R)$-bounded operator.  From \eqref{R I2} we obtain
\begin{align*}
\bigl\| \sqrt{R_0}\, \langle x\rangle^{-1} q \sqrt{R_0} \bigr\|_{\I_2(L^2(\R))}
	&\lesssim \kappa^{-1/2} \bigl\| \langle x\rangle^{-1} q \bigr\|_{H^{-1}_\kappa(\R)} \lesssim \kappa^{-1/2} \bigl\| q \bigr\|_{H^{-1}_\kappa(\R/\Z)}
\end{align*}
and similarly,
\begin{align*}
\bigl\| \sqrt{R_0}\, f\psi \sqrt{R_0} \bigr\|_{\I_2(L^2(\R))} + \bigl\| \sqrt{R_0}\, f\psi\langle x\rangle\sqrt{R_0} \bigr\|_{\I_2(L^2(\R))}
\lesssim \kappa^{-1/2} \bigl\| f \bigr\|_{H^{-1}_\kappa(\R/\Z)}.
\end{align*}
Combining the preceding immediately yields \eqref{A.1.2}.
\end{proof}

\begin{prop}\label{P:periodic 1}
Let $q\in H^{-1}(\R/\Z)$.  There is a unique self-adjoint operator $L$ acting on $L^2(\R)$ associated to the semi-bounded quadratic form
$$
\psi \mapsto \int_{\R} |\psi'(x)|^2 + q(x) |\psi(x)|^2\,dx .
$$
Furthermore, there exists $\delta>0$, so that if $q$ and $\kappa$ obey \eqref{periodic smallness}, then the resolvent $R:=(L+\kappa^2)^{-1}$ admits a continuous integral kernel $G(x,y;\kappa,q)$ given by the uniformly convergent series
\begin{align}\label{E:periodic G}
G(x,y;\kappa, q) = \tfrac{1}{2\kappa} e^{-\kappa|x-y|} + \sum_{\ell=1}^\infty (-1)^\ell \Bigl\langle \sqrt{R_0}\, \delta_x, \Bigl(\!\sqrt{R_0}\,q \sqrt{R_0}\Bigr)^\ell \sqrt{R_0}\,\delta_y\Bigr\rangle.
\end{align}
\end{prop}

\begin{proof}
Regarding the existence and uniqueness of $L$, we note that \eqref{A.1.1} guarantees that $q$ is an infinitesimally form bounded perturbation and then apply \cite[Theorem~X.17]{MR0493420}.  This is the same argument used in the proof of Proposition~\ref{P:sa L}.

Using Plancherel, it is easy to check that $x\mapsto \sqrt{R_0}\delta_x$ is H\"older-continuous as a map from $\R$ to $L^2(\R)$.  Thus, convergence of the series \eqref{E:periodic G} and continuity of the result follows whenever
$\sqrt{R_0}\,q \sqrt{R_0}$ is a contraction; this in turn follows from \eqref{periodic smallness} and \eqref{A.1.1} when $\delta$ is small enough.
\end{proof}

We define $g(x;\kappa,q)$ and $\rho(x;\kappa,q)$ exactly as in Section~\ref{S:2}; see \eqref{g defn} and \eqref{E:rho defn}.  Let us now demonstrate their basic properties:

\begin{prop}\label{P:periodic 2}
There exists $\delta>0$, so that the following are true for all $\kappa\geq 1:$\\
(i) The mappings
\begin{align}\label{periodic diffeos}
q\mapsto g-\tfrac1{2\kappa} \qtq{and} q\mapsto \kappa-\tfrac1{2g}
\end{align}
are (real analytic) diffeomorphisms of $B_{\delta,\kappa}$ into $H^1(\R/\Z)$.\\
(ii) For every $q\in B_{\delta,\kappa}$ and every integer $s\geq 0$,
\begin{align}\label{periodic stronger}
\|g'(x)\|_{H^{s}(\R/\Z)}  \lesssim_s \| q \|_{H^{s-1}(\R/\Z)} .
\end{align}
(iii) For every $q\in B_{\delta,\kappa}$, $\rho(x;q)$ is non-negative and in $H^1(\R/\Z)$. Moreover,
\begin{align}\label{periodic alpha}
\alpha(\kappa,q):=\int_0^1 \rho(x)\,dx \approx \kappa^{-1} \sum_{\xi\in 2\pi\Z} \frac{|\hat q(\xi)|^2}{\xi^2+4\kappa^2},
\end{align}
uniformly for $\kappa\geq 1$ and $q\in B_{\delta,\kappa}$.
\end{prop}

\begin{proof}
First we show that $g\in H^{1}(\R/\Z)$.  That it is periodic is self-evident from \eqref{E:periodic G}. To estimate its norm, we pick $\psi\in C^\infty_c(\R)$ so that
$$
\sum_{k\in\Z} \psi(x-k) \equiv 1.
$$
The utility of this partition of unity for us stems from the duality relation
$$
\| h \|_{H^{1}_\kappa(\R/\Z)} = \sup \biggl\{ \int_\R h(x) \psi(x) f(x)\,dx : f\in C^\infty(\R/\Z) \text{ and }\|f\|_{H^{-1}_\kappa(\R/\Z)}\leq 1\biggr\}.
$$

Now given $f\in C^\infty(\R/\Z)$, from \eqref{E:periodic G} we obtain that
\begin{align}\label{dual g}
\int \bigl[g(x) -\tfrac1{2\kappa}\bigr] \psi(x) f(x) \,dx = \sum_{\ell=1}^\infty (-1)^\ell \tr\Bigl\{ \!\sqrt{R_0}\,f\psi \sqrt{R_0} \Bigl(\!\sqrt{R_0}\,q \sqrt{R_0}\Bigr)^\ell \Bigr\}
\end{align}
and thence, using \eqref{A.1.2}, \eqref{A.1.1}, and \eqref{periodic B delta},
\begin{align}\label{periodic g in H1}
\bigl\|g(x) -\tfrac1{2\kappa}\bigr\|_{H^1(\R/\Z)} \lesssim \kappa^{-1} \|q\|_{H^{-1}(\R/\Z)},
\end{align}
provided $\delta$ is chosen sufficiently small.  Moreover, this argument shows that the first mapping in \eqref{periodic diffeos} is real-analytic by directly proving convergence of the power series.

When combined with \eqref{translation identity}, the estimates just presented also lead to a proof of \eqref{periodic stronger}; for further details see the proof of Proposition~\ref{P:diffeo}.

We consider now the inverse mapping.  As in Section~\ref{S:2},
$$
\kappa \cdot dg\bigr|_{q\equiv 0} = - R_0(2\kappa)
$$
is a unitary map of $H^{-1}_\kappa(\R/\Z)$ onto $H^1_\kappa(\R/\Z)$.  Thus, by the inverse function theorem, $q\mapsto g-\tfrac1{2\kappa}$ is a diffeomorphism in some neighbourhood of zero.  We must verify that the inverse mapping extends to the whole of $B_{\delta,\kappa}$.  Differentiating \eqref{dual g} with respect to $q$ and applying \eqref{A.1.1} and \eqref{A.1.2} yields
\begin{align}\label{periodic dg in H1}
\Bigl\| dg -dg\bigr|_{q\equiv0} \Bigr\|_{H^{-1}_\kappa\to H^{\vphantom{+}1}_\kappa}  \lesssim \kappa^{-3/2} \|q\|_{H^{-1}(\R/\Z)},
\end{align}
which suffices for this task.

The diffeomorphism property extends from the first map in \eqref{periodic diffeos} to the second by exactly the same argument presented in the proof of Proposition~\ref{P:diffeo}.

We turn now to part (iii), focussing on \eqref{periodic alpha}.  As previously, we proceed by computing derivatives, beginning with
\begin{align}\label{peri red}
\frac{d\ }{ds}\biggr|_{s=0} \int_0^1 \frac{dy}{2g(y;q+sf)} = \int_0^1 g(x) f(x)\,dx.
\end{align}
This may be proved as follows: By the resolvent identity, periodicity, and Lemma~\ref{L:D 1/g},
\begin{align*}
\text{LHS\eqref{peri red}} &= \int_0^1 \int_\R \frac{G(y,x)f(x)G(x,y)}{2g(y)^2}\,dx\,dy \\
&= \sum_{k\in\Z} \int_0^1 \int_0^1 \frac{G(y,x+k)f(x+k)G(x+k,y)}{2g(y)^2}\,dx\,dy \\
&= \sum_{k\in\Z} \int_0^1 \int_0^1 \frac{G(y-k,x)f(x)G(x,y-k)}{2g(y-k)^2}\,dx\,dy \\
&= \int_0^1 \int_\R \frac{G(y,x)G(x,y)}{2g(y)^2}\,dy\,f(x)\,dx = \text{RHS\eqref{peri red}}.
\end{align*}

Beginning with \eqref{peri red} and using \eqref{O35pp} shows
\begin{align*}
d^2\alpha\bigr|_{q\equiv 0} (f,f) &= \tfrac{1}{4\kappa^2} \int_0^1 \!\! \int_\R e^{-2\kappa|y-z|} f(y) f(z) \,dy\,dz
	= \tfrac{1}{\kappa} \sum_{\xi\in 2\pi\Z} \frac{|\hat f(\xi)|^2}{\xi^2+4\kappa^2}.
\end{align*}
Relying also on \eqref{periodic dg in H1} yields
\begin{align*}
\Bigl| d^2\alpha(f,f) - d^2\alpha\bigr|_{q\equiv 0} (f,f)\Bigr| & \lesssim \kappa^{-3/2} \|q\|_{H^{-1}(\R/\Z)} \|f\|_{H^{-1}_\kappa(\R/\Z)}^2.
\end{align*}
In this way, we see that the power series expansion of $\alpha$ as a function of $q$ is dominated by its quadratic term throughout \eqref{periodic B delta}, thus proving \eqref{periodic alpha}.
\end{proof}

Proposition~\ref{P:periodic 2} contains no analogue of \eqref{O37}.  Unlike in the decaying case, the quantity defined in \eqref{periodic alpha} does not coincide precisely with the renormalized perturbation determinant considered in \cite{KVZ}.  To describe the connection, we must introduce several new objects.  Henceforth, we consider only $q\in C^\infty(\R/\Z)$.  Once we have established suitable identities in this setting, they may be extended via analyticity to $q$ that are merely $H^{-1}(\R/\Z)$.

Let $\mcR_0$ denote the resolvent associated to the Laplacian on $[0,1]$ with periodic boundary conditions; concretely,
$$
\mcR_0 : \sum_{\xi\in2\pi\Z} c_\xi e^{i\xi x} \mapsto \sum_{\xi\in2\pi\Z} \frac{c_\xi e^{i\xi x}}{\xi^2 + \kappa^2}.
$$
Note that this coincides with $\mcM_0^2$ where $\mcM_0$ is as in the proof of \eqref{A.1.1}.  Using \eqref{mcM norm}, we see that the resolvent of the operator $\mathcal{L}=-\partial_x^2+q$, acting on $L^2([0,1])$ with periodic boundary conditions, can be expanded in a convergent series
$$
\mcR = \mcR_0 + \sum_{\ell=1}^\infty (-1)^\ell \sqrt{\mcR_0} \left(\sqrt{\mcR_0}\, q \, \sqrt{\mcR_0}\right)^\ell \sqrt{\mcR_0}
$$
whenever $\kappa$ and $q$ satisfy \eqref{periodic smallness} for suitable $\delta>0$.  Moreover, the kernels of these operators can be found by the method of images:
\begin{align}\label{mcR0 kernel}
\langle \delta_x, \mcR_0\delta_y\rangle = \sum_{k\in\Z} \langle \delta_x, R_0\delta_{y+k}\rangle =  \tfrac{1}{2\kappa(1-e^{-\kappa})} \bigl[ e^{-\kappa\|x-y\|} + e^{-\kappa(1-\|x-y\|)} \bigr],
\end{align}
where $\|x-y\|=\dist(x-y,\Z)$, and similarly,
\begin{align}\label{mcR kernel}
\mathcal G(x,y) := \langle \delta_x, \mcR \delta_y\rangle = \sum_{k\in\Z} G(x,y+k).
\end{align}

The last object we need to define is the Lyapunov exponent.  Let $\psi_\pm(x;\kappa)$ be the Weyl solutions introduced earlier in this section.  Due to the periodicity of $q$, we see that $x\mapsto\psi_+(x+1)$ and $x\mapsto\psi_-(x+1)$ constitute equally good Weyl solutions and so must differ from the originals by numerical constants.  Noting the constancy of the Wronskian as well as the square-integrability constraint, we see
that there is a $\gamma=\gamma(\kappa)>0$ so that
$$
\psi_+(x+1;\kappa) = e^{-\gamma(\kappa)} \psi_+(x;\kappa) \qtq{and} \psi_-(x+1;\kappa) = e^{+\gamma(\kappa)} \psi_-(x;\kappa).
$$
This quantity $\gamma$ is known as the Lyapunov exponent.  Employing these relations to sum in \eqref{mcR kernel}, we deduce that
\begin{align}\label{mcR kernel'}
\mathcal G(x,x) = \frac{1+e^{-\gamma}}{1-e^{-\gamma}} G(x,x).
\end{align}

\begin{prop}  For $q\in C^\infty(\R/\Z)$ and $\kappa$ satisfying \eqref{periodic smallness}, we have
\begin{align}\label{Lyapunov formula}
  \gamma(\kappa) = \int_0^1 \frac{dx}{2g(x)},
\end{align}
\begin{align}\label{periodic trace}
\tr\left(\sqrt{\mcR_0}\, q \, \sqrt{\mcR_0} \right) = \tfrac{1+e^{-\kappa}}{1-e^{-\kappa}}  \int_0^1  \bigl[\tfrac{1}{2} e^{-2\kappa|\cdot|}*q\bigr](x) \,dx
    = \tfrac1{2\kappa}\,\tfrac{1+e^{-\kappa}}{1-e^{-\kappa}} \int_0^1  q(x) \,dx,
\end{align}
which is a Casimir, and
\begin{align}\label{periodic determinant}
  \log\det\left( 1 + \sqrt{\mcR_0}\, q \, \sqrt{\mcR_0} \right) = \log\bigl(e^{\gamma} - 2 + e^{-\gamma} \bigr) - \log\bigl(e^{\kappa} - 2 + e^{-\kappa} \bigr).
\end{align}
Here the trace and determinant are with respect to the Hilbert space $L^2(\R/\Z)$.
\end{prop}

\begin{proof}
The proof of \eqref{Lyapunov formula} is very simple: combining $g(x)=\psi_+(x)\psi_-(x)$ with the Wronskian relation \eqref{E:Wron}, we have
\begin{align*}
\int_0^1 \frac{dx}{2g(x)} = \tfrac12 \int_0^1 \frac{d }{dx} \log\Bigl[\tfrac{\psi_-(x)}{\psi_+(x)}\Bigr]\,dx = \tfrac12 \log\Bigl[\tfrac{\psi_-(x+1)\psi_+(x)}{\psi_-(x)\psi_+(x+1)}\Bigr] = \gamma.
\end{align*}

By \eqref{mcR0 kernel}, we have
$$
\tr\Bigl\{ \sqrt{\mcR_0}\, q \, \sqrt{\mcR_0} \Bigr\} = \frac{(1+e^{-\kappa})}{2\kappa(1-e^{-\kappa})} \int_0^1 q(x)\,dx,
$$
while
\begin{align*}
\int_0^1\int_\R \tfrac{1}{2} e^{-2\kappa|x-y|} q(y)\,dy\,dx &= \sum_{k\in\Z} \int_0^1\int_0^1 \tfrac{1}{2} e^{-2\kappa|x-k-y|} q(y)\,dy\,dx \\
&= \int_0^1\int_\R \tfrac{1}{2} e^{-2\kappa|x-y|} q(y) \,dx\,dy = \tfrac{1}{2\kappa} \int_0^1 q(y)\,dy.
\end{align*}
This proves \eqref{periodic trace}.

The identity \eqref{periodic determinant} can be readily deduced from \cite[Theorem~2.9]{MR0559928}, which is a recapitulation of venerable results of Hill and of Wittaker and Watson.  For completeness, we given an alternate proof paralleling our arguments from the rapidly decreasing case.

By \eqref{Lyapunov formula}, we see that \eqref{periodic determinant} holds in the case $q\equiv 0$.  Moreover, arguing as in the decaying case, we find that for any $f\in C^\infty(\R/\Z)$,
$$
\frac{d\ }{ds}\biggr|_{s=0} \log\det\left( 1 + \sqrt{\mcR_0}\, (q+sf) \, \sqrt{\mcR_0} \right) (x) = \int_0^1 \mathcal G(x,x) f(x)\,dx.
$$
On the other hand, by \eqref{Lyapunov formula} and \eqref{peri red},
$$
\frac{d\ }{ds}\biggr|_{s=0} \log\bigl(e^{\gamma(\kappa;q+sf)} - 2 + e^{-\gamma(\kappa;q+sf)} \bigr) = \frac{e^\gamma-e^{-\gamma}}{(e^{\gamma} - 2 + e^{-\gamma})}\int_0^1 g(x) f(x)\,dx.
$$
In view of \eqref{mcR kernel'}, these two derivatives agree.  Thus equality in \eqref{periodic determinant} extends to all $q\in B_{\delta,\kappa}$.
\end{proof}

\begin{corollary}  For smooth initial data, the conservation of
$$
\int_0^1 \rho(x)\,dx
$$
under the KdV flow, which follows from \eqref{E:l5.1h}, is equivalent to conservation of
$$
-\log\det_2\left( 1 + \sqrt{\mcR_0}\, q \, \sqrt{\mcR_0} \right),
$$
which was proved in \cite{KVZ}.
\end{corollary}

\section{Local smoothing}\label{S:7}

Our first goal in this section is to derive a local smoothing result for $H^{-1}$-solutions to \eqref{KdV} on the line.  A similar a priori estimate was obtained by Buckmaster and Koch in \cite{MR3400442} via the Miura map.

\begin{lemma}[Local smoothing]\label{L:loc smoothing}
There exists $\delta>0$ so that for every $H^{-1}(\R)$-solution $q(t)$ to \eqref{KdV}, in the sense of Corollary~\ref{C:1}, with initial data $q(0)\in B_\delta$,
\begin{equation}\label{E:loc smoothing *}
\sup_{t_0,x_0\in\R} \ \int_{0}^{1} \!\! \int_{0}^{1} |q(t-t_0,x-x_0)|^2 \,dx\,dt \lesssim \delta^2.
\end{equation}
\end{lemma}

\begin{proof}
As noted already in Proposition~\ref{L:5.2}, conservation of $\alpha(\kappa=1)$ guarantees that
\begin{align*}
\| q \|_{L^\infty_x H^{-1}_x} ^2 \lesssim \delta^2.
\end{align*}
This allows us to choose $\delta$ sufficiently small that all results from Section~\ref{S:3} can be applied at all times $t\in\R$.
It also means that it suffices to prove \eqref{E:loc smoothing *} with $t_0=x_0=0$.

Let us now fix a smooth function $\phi$ whose derivative $\phi'$ is positive and Schwartz and define
$$
\psi(x) = \tfrac32 \int_\R e^{-2|x-y|} \phi'(y) \,dy,
$$
which is positive everywhere.

Suppose first that $q(t)$ is a Schwartz solution to \eqref{KdV}.  Setting $\kappa=1$ in \eqref{E:l5.1h}, we obtain
\begin{align*}
\tfrac{d\ }{dt}  \rho(t,x) &= \Bigl( \tfrac32 \bigl[e^{-2|\cdot|}* q(t)^2\bigr](x)  + 2q(t,x)\bigl[ 1 - \tfrac{1}{2g(t,x)}\bigr] - 4\rho(t,x) \Bigr)'.
\end{align*}
Integrating this against $\phi(x)$ and integrating by parts, yields
\begin{align}
\int_0^1 \! \int_\R |q(t,x)|^2 \psi(x) \,dx\,dt &= \int_\R [\rho(0,x)-\rho(1,x)] \phi(x) \,dx \notag\\
&\quad - 2 \int_0^1 \! \int_\R q(t,x)\bigl[ 1 - \tfrac{1}{2g(t,x)}\bigr] \phi'(x) \,dx\,dt \label{64}\\
&\quad + 4 \int_0^1 \! \int_\R \rho(t,x) \phi'(x) \,dx\,dt . \notag
\end{align}
But by the results of Section~\ref{S:3}, the right-hand side is bounded uniformly;  thus \eqref{E:loc smoothing *} follows for Schwartz solutions.

Next we allow $q(t)$ to be a general (not Schwartz) solution to \eqref{KdV} and suppose $q_n(t)$ is a sequence of Schwartz solutions with $q_n(0)\to q(0)$ in $H^{-1}(\R)$.  By weak lower-semicontinuity of the $L^2$-norm and the fact that weak convergence is guaranteed by Theorem~\ref{T:converge},
\begin{equation*}
\int_{0}^{1} \int_{0}^{1} |q(t,x)|^2 \,dx\,dt \leq \liminf_{n\to\infty} \int_{0}^{1} \int_0^1 |q_n(t,x)|^2 \,dx\,dt.
\end{equation*}
Thus, \eqref{E:loc smoothing *} for such general solutions $q(t)$ follows from the Schwartz-class case already proven.
\end{proof}

Using this a priori bound as a stepping stone, we will now show that solutions whose initial data converge in $H^{-1}(\R)$ actually converge in the local smoothing norm as claimed in Theorem~\ref{T:ls conv}.  In fact, the following proposition is strictly stronger than this theorem because of the additional uniformity in $x_0$.

\begin{prop}\label{P:loc smoothing}
Let $q(t)$ and $q_n(t)$ be $H^{-1}(\R)$-solutions to \eqref{KdV}, in the sense of Corollary~\ref{C:1}, with initial data $q_n(0)\to q(0)$ in $H^{-1}(\R)$.  Then for every $T>0$,
\begin{equation}\label{E:loc smoothing n}
\lim_{n\to\infty} \ \sup_{x_0\in\R} \ \int_{-T}^T \int_0^1 |q(t,x-x_0)-q_n(t,x-x_0)|^2 \,dx\,dt  = 0.
\end{equation}
In particular, solutions in the sense of Corollary~\ref{C:1} are distributional solutions.
\end{prop}

A major part of the argument leading to Proposition~\ref{P:loc smoothing} is a refinement of the proof of Lemma~\ref{L:loc smoothing}. The key improvement stems from analyzing the behavior of the various terms in \eqref{64} as $\kappa\to\infty$, rather than simply setting $\kappa=1$.  We begin with the following preliminary estimates:

\begin{lemma}
Fix $\psi\in C^\infty_c(\R)$ with $\supp(\psi)\subset (0,1)$.  There exists $\delta>0$ so that
\begin{gather}
\bigl\| \psi(x)\bigl[g(x)-\tfrac1{2\kappa}\bigr] + \kappa^{-1}[R_0(2\kappa)(q\psi)](x) \bigr\|_{L^2(\R)}^2
	\lesssim \kappa^{-7} \biggl[ 1+ \int_0^1 |q(x)|^2\,dx \biggr] \label{E:psi g}\\
\bigl\| \psi(x)\bigl[\kappa -\tfrac1{2g(x)}\bigr] + 2\kappa [R_0(2\kappa)(q\psi)](x) \bigr\|_{L^2(\R)}^2
	\lesssim \kappa^{-3} \biggl[ 1+ \int_0^1 |q(x)|^2\,dx \biggr] \label{E:psi 1/g}\\
\ \ \biggl| \int \rho(x) \psi(x)^2 \,dx - \tfrac{1}{2\kappa} \int \frac{|\widehat{q\psi} (\xi)|^2\,d\xi}{\xi^2+4\kappa^2} \biggr|
	\lesssim \kappa^{-7/2} \biggl[ 1+ \int_0^1 |q(x)|^2\,dx \biggr] \label{E:psi rho}\\
\biggl| \iint q(x)^2 \kappa e^{-2\kappa|x-y|} \psi(y)^2 \,dx\,dy - \int q(x)^2 \psi(x)^2 \,dx \biggr|
	\lesssim \int_\R \frac{|q(x)|^2\,dx}{\kappa(1+x^2)} \label{E:psi dumb}
\end{gather}
for every $q\in B_\delta$ and $\kappa\geq 1$.  (Note that the implicit constants depend on $\psi$.)
\end{lemma}

\begin{proof}
We begin with a commutator calculation:
\begin{align*}
[\psi(x),R_0] &= R_0 \bigl(-2\partial_x \psi'(x) + \psi''(x)\bigr) R_0 \\
&= R_0 \bigl(-2\partial_x\bigr) [\psi'(x),R_0] + R_0 \bigl(-2\partial_x\bigr)R_0\psi'(x) + R_0\psi''(x) R_0.
\end{align*}
This shows that for $\kappa\geq 1$, we can write
\begin{equation}\label{psi comm bound}
\begin{gathered}\relax
[\psi(x),R_0] = \sqrt{R_0} A \sqrt{R_0} = \sqrt{R_0} B \sqrt{R_0} + \sqrt{R_0} C \sqrt{R_0} \psi'(x) \\
\text{with} \quad  \|A\|_{L^2\to L^2} + \|C\|_{L^2\to L^2} \lesssim \kappa^{-1} \qtq{and} \|B\|_{L^2\to L^2} \lesssim \kappa^{-2}.
\end{gathered}
\end{equation}

From the series \eqref{E:g series}, we have
\begin{align}
&\int_\R\{ \psi(x)\bigl[g(x)-\tfrac1{2\kappa}\bigr] + \kappa^{-1}[R_0(2\kappa)(q\psi)](x) \bigr\} f(x)\,dx \label{dual psi g}\\
={}&  \sum_{\ell\geq 2} (-1)^\ell \tr\Bigl\{ \sqrt{R_0}\,f\,\sqrt{R_0} \sqrt{R_0}\,\psi q\,\sqrt{R_0} \Bigl( \sqrt{R_0}\,q\,\sqrt{R_0}\Bigr)^{\ell-1}  \Bigr\} \notag\\
&\ +\sum_{\ell\geq 1} (-1)^\ell  \tr\Bigl\{ \sqrt{R_0}\,f\,\sqrt{R_0} B \Bigl( \sqrt{R_0}\,q\,\sqrt{R_0}\Bigr)^\ell  \Bigr\} \notag\\
&\ +\sum_{\ell\geq 1} (-1)^\ell \tr\Bigl\{ \sqrt{R_0}\,f\,\sqrt{R_0} C \sqrt{R_0}\,\psi' q\,\sqrt{R_0} \Bigl( \sqrt{R_0}\,q\,\sqrt{R_0}\Bigr)^{\ell-1}  \Bigr\}. \notag
\end{align}
Using
\begin{align}\label{I2 S6}
\Bigl\| \sqrt{R_0}\,h\,\sqrt{R_0} \Bigr\|_{\I_2} \lesssim \kappa^{-3/2} \| h \|_{L^2}
\qtq{and} \Bigl\| \sqrt{R_0}\,q\,\sqrt{R_0} \Bigr\|_{\I_2} \lesssim \kappa^{-1/2} \| q \|_{H^{-1}},
\end{align}
which follow from \eqref{R I2}, we then deduce that
\begin{align*}
\text{LHS\eqref{dual psi g}} \leq \kappa^{-7/2} \|f\|_{L^2} \Bigl\{ \|\psi q\|_{L^2} + 1 + \|\psi' q\|_{L^2}\Bigr\},
\end{align*}
provided, say, $\delta\leq\frac12$.  This proves \eqref{E:psi g}.  For future use, we also note that with the aid of \eqref{E:psi g}, one may readily show
\begin{align}\label{step to rho2}
\int \bigl(g(x) - \tfrac1{2\kappa}\bigr)^2\psi(x)^2\,dx + \bigl\| \kappa^{-1} R_0(2\kappa)(q\psi) \bigr\|_{L^2(\R)}^2
	\lesssim \kappa^{-6} \biggl[ 1+ \int_0^1 |q(x)|^2\,dx \biggr] .
\end{align}

Next we prove \eqref{E:psi 1/g}.  This almost follows from \eqref{E:psi g}; indeed, writing
\begin{align}\label{1/g expansion}
\kappa -\tfrac1{2g(x)} = 2\kappa^2\bigl( g(x) - \tfrac1{2\kappa}\bigr) - \tfrac{2\kappa^2}{g(x)} \bigl( g(x) - \tfrac1{2\kappa}\bigr)^2
\end{align}
and invoking \eqref{E:psi g}, we are left only to prove
\begin{align}\label{psi/g left}
\int \tfrac{2\kappa^2}{g(x)} \bigl( g(x) - \tfrac1{2\kappa}\bigr)^2 \psi(x)^2\,dx \lesssim \kappa^{-3} \biggl[ 1+ \int_0^1 |q(x)|^2\,dx \biggr].
\end{align}
This then follows from \eqref{step to rho2}.

We begin the proof of \eqref{E:psi rho} by expanding one step further than \eqref{1/g expansion} to write
\begin{equation*}
\begin{gathered}
\rho(x) =\sum_{i=1}^3 \rho_i(x) \qtq{with} \rho_1(x):= 2\kappa^2\bigl\{g(x) - \tfrac1{2\kappa} + \tfrac1\kappa [R_0(2\kappa)q](x)\bigr\},\\
\rho_2(x):=  - 4\kappa^3\bigl(g(x) - \tfrac1{2\kappa}\bigr)^2 \qtq{and} \rho_3(x):= 4\kappa^3\bigl(g(x) - \tfrac1{2\kappa}\bigr)^3 / g(x).
\end{gathered}
\end{equation*}

Let us begin our analysis with the contribution of $\rho_1$.  From \eqref{E:g series}, we have
\begin{align*}
\int \rho_1(x)\psi(x)^2\,dx = 2\kappa^2 \sum_{\ell\geq 2} (-1)^\ell \tr\Bigl\{ \psi(x)^2 \sqrt{R_0}\Bigl(\sqrt{R_0}\,q\,\sqrt{R_0}\Bigr)^\ell \sqrt{R_0}\Bigr\}.
\end{align*}
Now for any $\ell\geq2$, we have from \eqref{I2 S6}, \eqref{psi comm bound} and its adjoint that
\begin{align*}
&\tr \Bigl\{ \psi(x)^2 \sqrt{R_0} \Bigl(\sqrt{R_0}\,q\,\sqrt{R_0}\Bigr)^\ell \sqrt{R_0} \Bigr\} \\
&\ \ = \tr \Bigl\{  \sqrt{R_0}\,\psi q R_0^2 \psi q\,\sqrt{R_0} \Bigl(\sqrt{R_0}\,q\,\sqrt{R_0}\Bigr)^{\ell-2} \Bigr\}
	+ O\biggl( \kappa^{-6} \delta^{\ell-2} \biggl[ 1+ \int_0^1 |q(x)|^2\,dx \biggr] \biggr) .
\end{align*}
The error term here sums acceptably over $\ell\geq 2$.  The contribution of the first term is also acceptable provided we restrict to $\ell\geq 3$; indeed,
$$
\Bigl| \tr \Bigl\{ \sqrt{R_0}\,\psi q R_0^2 \psi q\,\sqrt{R_0} \Bigl(\sqrt{R_0}\,q\,\sqrt{R_0}\Bigr)^{\ell-2} \Bigr\} \Bigr|
	\lesssim \kappa^{-5-\frac{\ell-2}2} \delta^{\ell-2} \int_0^1 |q(x)|^2\,dx.
$$
Combining all of this, we deduce that
\begin{align*}
\int \rho_1(x)\psi(x)^2\,dx = 2\kappa^2 \tr \Bigl\{ R_0\psi q R_0 \psi q R_0 \Bigr\}
	+ O\biggl( \kappa^{-7/2} \biggl[ 1+ \int_0^1 |q(x)|^2\,dx \biggr] \biggr) .
\end{align*}

We turn now to $\rho_2$.  Combining \eqref{step to rho2} and \eqref{E:psi g} gives
\begin{align*}
\biggl| \int \bigl(g(x) - \tfrac1{2\kappa}\bigr)^2\psi(x)^2\,dx - \kappa^{-2} \bigl\| R_0(2\kappa)(q\psi) \bigr\|_{L^2(\R)}^2 \biggr|
	\lesssim \kappa^{-13/2} \biggl[ 1+ \int_0^1 |q(x)|^2\,dx \biggr].
\end{align*}
Therefore,
\begin{align*}
\int \rho_2(x)\psi(x)^2\,dx = - 4 \kappa \bigl\| R_0(2\kappa)(q\psi) \bigr\|_{L^2(\R)}^2
	+ O\biggl( \kappa^{-7/2} \biggl[ 1+ \int_0^1 |q(x)|^2\,dx \biggr] \biggr) .
\end{align*}

We now consider $\rho_3$.  From \eqref{R I2} we have
\begin{align*}
\Bigl\| \sqrt{R_0}\,h\,\sqrt{R_0} \Bigr\|_{\I_2}^2 \lesssim \kappa^{-1} \| h \|_{L^1}^2 \int\frac{d\xi}{\xi^2+4\kappa^2} \lesssim \kappa^{-2} \| h \|_{L^1}^2 .
\end{align*}
Employing this to estimate the series \eqref{E:g series} via duality, we obtain
$$
\bigl\| g - \tfrac1{2\kappa} \bigr\|_{L^\infty} \lesssim \kappa^{-3/2} \delta .
$$
Combining this with \eqref{step to rho2} yields
\begin{align*}
\int \rho_3(x)\psi(x)^2\,dx \lesssim \kappa^{-7/2} \biggl[ 1+ \int_0^1 |q(x)|^2\,dx \biggr] .
\end{align*}

To derive the claim \eqref{E:psi rho} by combining our results on each part of $\rho$, we need one additional identity, namely,
$$
2\kappa^2 \tr \Bigl\{ R_0\psi q R_0 \psi q R_0 \Bigr\} - 4 \kappa \bigl\| R_0(2\kappa)(q\psi) \bigr\|_{L^2(\R)}^2
	= \tfrac{1}{2\kappa} \int \frac{|\widehat{q\psi} (\xi)|^2\,d\xi}{\xi^2+4\kappa^2} .
$$
That this equality holds follows from the same calculation we carried out in \eqref{delta2alpha}.

The last estimate \eqref{E:psi dumb} is relatively trivial.  As $\int \kappa e^{-2\kappa|x-y|}\,dy =1$,
\begin{align*}
\biggl| \psi(x)^2 - \int \kappa e^{-2\kappa|x-y|} \psi(y)^2 \,dy \biggr| \lesssim \int \kappa e^{-2\kappa|y-x|} |x-y| \,dy \lesssim \kappa^{-1},
\end{align*}
which settles the case $|x|\leq 10$.  For $|x|\geq 10$, we have
\begin{align*}
\biggl| \psi(x)^2 - \int \kappa e^{-2\kappa|x-y|} \psi(y)^2 \,dy \biggr| = \int_0^1 \kappa e^{-2\kappa|y-x|} \psi(y)^2 \,dy	\lesssim \kappa e^{-\kappa|x|} ,
\end{align*}
which offers more than enough decay in both $x$ and $\kappa$.
\end{proof}

\begin{lemma}\label{L:kappa ls}
There is a $\delta>0$ so that the following is true: Let $Q$ be a family of Schwartz solutions to \eqref{KdV} on the line such that $\{q(0):q\in Q\}$ is an equicontinuous subset of $B_\delta$.  Then, for any $\psi\in C^\infty_c(\R)$ and any $T>0$, we have
\begin{equation}\label{E:loc smoothing hi}
\lim_{\kappa\to\infty} \ \sup_{q\in Q} \ \int_{-T}^T \int \frac{\xi^2 |\widehat{q\psi} (t,\xi)|^2\,d\xi}{\xi^2+4\kappa^2} \,d\xi\,dt  = 0.
\end{equation}
\end{lemma}

\begin{proof}
Without loss of generality, we may assume that $\supp(\psi)\subset (0,1)$.  Throughout the proof, we regard $\psi$ and $T$ as fixed and implicit constants may depend on them.   Multiplying \eqref{E:l5.1h} by $\kappa$ and integrating against
$$
\phi(x) = \int_{-\infty}^x \psi(y)^2\,dy,
$$
we obtain
\begin{align}
\int \kappa [\rho(T,x)-\rho(-T,x)] & \phi(x) \,dx \label{lsk1}\\
&= -\tfrac32 \int_{-T}^T \! \iint |q(t,x)|^2 \kappa e^{-2\kappa|x-y|}\psi(y)^2 \,dx\,dy\,dt \label{lsk2}\\
&\quad - 2\kappa \int_{-T}^T \! \int q(t,x)\bigl[ \kappa - \tfrac{1}{2g(t,x)}\bigr] \psi(x)^2 \,dx\,dt \label{lsk3}\\
&\quad + 4\kappa^3 \int_{-T}^T \! \int \rho(t,x) \psi(x)^2 \,dx\,dt.\label{lsk4}
\end{align}
We will discuss these terms one at a time.

From Proposition~\ref{P:Intro rho}, we have
$$
\bigl| \text{\eqref{lsk1}} \bigr| \lesssim \kappa \alpha(\kappa;q),
$$
which converges to zero as $\kappa\to\infty$ uniformly for $q\in Q$; see Lemma~\ref{L:equi 2}.

Combining \eqref{E:psi dumb} and Lemma~\ref{L:loc smoothing} yields
\begin{align*}
\biggl| \eqref{lsk2} + \tfrac32 \int_{-T}^T \! \int |q(t,x)|^2 \psi(x)^2 \,dx\,dt \biggr| \lesssim \kappa^{-1},
\end{align*}
or equivalently, by Plancherel,
\begin{align*}
\biggl| \eqref{lsk2} + \tfrac32 \int_{-T}^T \! \int |\widehat{\psi q}(t,\xi)|^2 \,d\xi\,dt \biggr| \lesssim \kappa^{-1}.
\end{align*}

From \eqref{E:psi 1/g} and Lemma~\ref{L:loc smoothing}, we have
\begin{align*}
\biggl| \eqref{lsk3} - \int_{-T}^T \! \iint \psi(x)q(t,x) \kappa e^{-2\kappa|x-y|} q(t,y)\psi(y) \,dx\,dy\,dt \biggr| \lesssim \kappa^{-1/2},
\end{align*}
or equivalently (see \eqref{R resolvent}),
\begin{align*}
\biggl| \eqref{lsk3} - \int_{-T}^T \! \int \frac{4\kappa^2|\widehat{\psi q}(t,\xi)|^2}{\xi^2+4\kappa^2} \,d\xi\,dt \biggr| \lesssim \kappa^{-1/2}.
\end{align*}

From \eqref{E:psi rho} and Lemma~\ref{L:loc smoothing},
$$
\biggl| \eqref{lsk4} - \tfrac{1}{2} \int_{-T}^T \! \int \frac{4\kappa^2|\widehat{\psi q}(t,\xi)|^2}{\xi^2+4\kappa^2} \,d\xi\,dt \biggr|
	\lesssim \kappa^{-1/2}.
$$

The claim \eqref{E:loc smoothing hi} now follows from recombining \eqref{lsk1}--\eqref{lsk4}.
\end{proof}

\begin{proof}[Proof of Proposition~\ref{P:loc smoothing}]
By the same scaling argument as in Theorem~\ref{T:converge}, it suffices to prove \eqref{E:loc smoothing n} for solutions that are small in $L^\infty_t H^{-1}_x$.  Moreover, we may assume that $q_n$ are Schwartz solutions, since Proposition~\ref{P:loc smoothing} in this reduced generality provides precisely the tool necessary to obtain the full version by approximation.

Let us fix $\psi\in C^\infty_c(\R)$, not identically zero, with $\supp(\psi)\subseteq(0,1)$.  By a simple covering argument, it suffices to show that
\begin{align}\label{E:ls 1}
\lim_{n\to\infty} \ \sup_{x_0\in\R} \ \int_{-T}^T \Bigl\|\bigl[q_n(t,x)-q(t,x)\bigr] \psi(x+x_0) \Bigr\|_{L^2(\R)}^2 \,dt  = 0.
\end{align}
We will do this by breaking into high- and low-frequency components, using a refined local smoothing argument to handle the former, and applying Theorem~\ref{T:converge} to handle the latter. The frequency decomposition is based on the multipliers
$$
m_\text{hi}(\xi) = \frac{|\xi|}{\sqrt{\xi^2+4\kappa^2}} \qtq{and}  m_\text{lo}(\xi) = \sqrt{1 - m_\text{hi}(\xi)^2} = \frac{2\kappa}{\sqrt{\xi^2+4\kappa^2}}.
$$

We begin with the low frequencies.  For $\kappa$ fixed, Theorem~\ref{T:converge} implies
\begin{align}
\lim_{n\to\infty} &\ \sup_{x_0\in\R}\ \int_{-T}^T  \; \Bigl\|m_\text{lo}(-i\partial_x) \Bigl(\bigl[q_n(t)-q(t)\bigr] \psi(\cdot + x_0) \Bigr)\Bigr\|_{L^2(\R)}^2 \,dt\notag\\
&\lesssim \kappa T  \lim_{n\to\infty}\ \sup_{x_0\in\R}\  \Bigl\| \bigl[q_n(t,x)-q(t,x)\bigr] \psi(x+x_0) \Bigr\|_{L^\infty_t H^{-1}_x ([-T,T]\times\R)}\label{ls  low} \\
&\lesssim \kappa T \|\psi\|_{H^1(\R)}  \lim_{n\to\infty} \bigl\| q_n(t,x)-q(t,x) \bigr\|_{L^\infty_t H^{-1}_x ([-T,T]\times\R)} =0.\notag
\end{align}

We turn now to the high-frequency part.  As the sequence $q_n(0)$ is convergent in $H^{-1}(\R)$, it is equicontinuous there.  Proposition~\ref{P:equi} then guarantees that $\{q_n(t) : t\in\R\text{ and } n\in\N\}$ is also $H^{-1}(\R)$-equicontinuous.  Thus Lemma~\ref{L:kappa ls} implies
\begin{equation}\label{ls n high}
\lim_{\kappa\to\infty} \ \sup_{n} \int_{-T}^T \; \Bigl\|m_\text{hi}(-i\partial_x) [q_n(t)\psi(\cdot+x_0)] \Bigr\|_{L^2(\R)}^2 \,dt = 0.
\end{equation}
Note that by Theorem~\ref{T:converge} and weak lower-semicontinuity, it then follows that
\begin{equation}\label{ls q high}
\lim_{\kappa\to\infty} \int_{-T}^T  \; \Bigl\|m_\text{hi}(-i\partial_x) [q(t)\psi(\cdot+x_0)] \Bigr\|_{L^2(\R)}^2 \,dt = 0.
\end{equation}

We are now ready to put the pieces together.  From \eqref{ls n high} and \eqref{ls q high}, we see that we can make the high-frequency contribution to LHS\eqref{E:ls 1} small, uniformly in $n$, by choosing $\kappa$ sufficiently large.  But then by \eqref{ls low}, we may make the low-frequency contribution as small as we wish by choosing $n$ sufficiently large.  This proves \eqref{E:ls 1} and so \eqref{E:loc smoothing n}.

Lastly, integration by parts shows that Schwartz solutions are distributional solutions of the initial-value problem, which is to say
\begin{align*}
\int h(0,x)q(0,x)\,dx + \int_0^\infty \!&\! \int [\partial_t h](t,x) q(t,x) \,dx\,dt \\
& =  \int_0^\infty \!\! \int - [\partial_x^3 h](t,x) q(t,x) + 3[\partial_x h](t,x) q(t,x)^2 \,dx\,dt
\end{align*}
for every $h\in C^\infty_c(\R\times\R)$ (as well as the analogous statement backwards in time).  This now extends to $H^{-1}$ solutions via Corollary~\ref{C:1} and Proposition~\ref{P:loc smoothing}.
\end{proof}

\appendix
\section{5th order KdV}\label{S:A}

The first polynomial conservation law for \eqref{KdV} beyond those discussed in the body of the paper is
\begin{align*}
H_\text{5th}(q) := \int \tfrac12 q''(x)^2 + 5 q(x)q'(x)^2 + \tfrac52 q(x)^4\,dx,
\end{align*}
which generates the dynamics
\begin{align}\label{5th}
\ddt q = q^{(5)} -20 q'q'' - 10 qq''' +30 q^2q' = \bigl(q^{(4)} - 10qq'' - 5 [q']^2+10q^3\bigr)'
\end{align}
under the Poisson structure~\eqref{3.0}.

Local well-posedness of \eqref{5th} in $H^s(\R)$ for $s\geq \tfrac54$ was shown in \cite{MR3096990,MR3301874}, which immediately implies global well-posedness in the energy space $H^2(\R)$; see also \cite{MR2653659} for a thorough discussion of results in Fourier--Lebesgue spaces.  As reviewed therein, these papers represent the culmination of a considerable body of prior work.  In particular, we do not know of any further progress on the low regularity problem on the line.  On the torus, however, an optimal well-posedness result was obtained in  \cite{MR3739929}.  Specifically, they prove that \eqref{5th} is well-posed in $L^2(\R/\Z)$ and moreover, that the solution map does not admit a continuous extension to $H^s(\R/\Z)$ for any $s<0$.  Our goal in this appendix is to show that a direct adaptation of the methods introduced in this paper yields an analogous well-posedness result on the line:

\begin{theorem}\label{T:5th}
Equation \eqref{5th} is globally well-posed in $L^2(\R)$.
\end{theorem}

As the proof of this theorem follows closely the treatment of \eqref{KdV} in the main body of the paper, we shall focus on recording the central estimates and identities, rather than recapitulating the arguments.

By the rescaling argument presented earlier and the fact that \eqref{5th} preserves the $L^2$ norm, it suffices to show well-posedness of this equation in
$$
B_\delta:=\{ q\in L^2(\R) : \| q\|_{L^2} \leq \delta\} \qquad\text{(endowed with the $L^2$ topology)}
$$
for some $\delta>0$.  As this is a subset of the ball \eqref{B delta}, all results of Section~\ref{S:2} and~\ref{S:3} still apply. Nevertheless some updating is required by the shift in regularity considered here:

\begin{prop}\label{P:A2}
There exists $\delta>0$ so that the following are true:\\
\emph{(a)} For all $\kappa\geq 1$, the two mappings in \eqref{diffeos} are diffeomorphisms from $B_\delta$ onto a neighbourhood of the origin in $H^{+2}(\R)$.\\
\emph{(b)} For all $q\in B_\delta$,
\begin{align}\label{A5tozero}
16\kappa^5\bigl(g-\tfrac1{2\kappa}\bigr) + 4\kappa^2 q \longrightarrow -q'' + 3q^2 \qquad\text{in $H^{-2}(\R)$ as $\kappa\to\infty$.}
\end{align}
Moreover, convergence is uniform on $L^2$-bounded and equicontinuous sets.
\end{prop}

\begin{proof}
As in the main body of the paper, the statements follow readily from certain basic estimates on the terms in the series \eqref{E:g series}.  As we must delve more deeply into the series than previously, we adopt the following notation
\begin{align}\label{A6}
h_\ell(x;\kappa) := (-1)^\ell  \langle \delta_x, R_0 (qR_0)^\ell \delta_x\rangle \qtq{so that} g(x;\kappa)-\tfrac1{2\kappa} = \sum_{\ell=1}^\infty h_\ell(x;\kappa).
\end{align}

The same computations that yielded \eqref{R I2} also show that
\begin{align*}
\hat{h}_1(\xi) =  - \frac{\hat q(\xi)}{\kappa(\xi^2+4\kappa^2)}.
\end{align*}
Easy consequences of this include
\begin{equation}\label{E:h1}
\| h_1 \|_{H^2_\kappa} = \kappa^{-1} \|q\|_{L^2}
\end{equation}
and
\begin{equation}\label{E:h1conv}
16\kappa^5h_1 + 4\kappa^2 q + q'' \longrightarrow 0 \quad\text{in $H^{-2}$ as $\kappa\to\infty$;}
\end{equation}
moreover, this convergence is uniform on $L^2$-bounded and equicontinuous sets.

Elementary, though rather tiresome, calculations show that
$$
\hat{h}_2(\xi) = \frac{1}{2\kappa\sqrt{2\pi}} \int \frac{(\xi-\eta)^2 + \eta^2 + \xi^2+24\kappa^2}{[\xi^2+4\kappa^2][\eta^2+4\kappa^2][(\xi-\eta)^2+4\kappa^2]} \hat q(\xi-\eta)\hat q(\eta)\,d\eta,
$$
which should be compared with
$$
\widehat{\,q^2\,}\!(\xi) = \frac{1}{\sqrt{2\pi}} \int \hat q(\xi-\eta)\hat q(\eta)\,d\eta.
$$
Indeed, expanding out the difference in partial fractions reveals that as $\kappa\to\infty$,
\begin{equation}\label{E:h2conv}
16\kappa^5 h_2 - 3q^2 \longrightarrow 0 \quad\text{in $L^1(\R)$, and so also in $H^{-2}(\R)$.}
\end{equation}
Moreover, convergence (even in $L^1$) is uniform for $q$ belonging to $L^2$-bounded and equicontinuous sets.

For any $\ell\geq 2$ and any Schwartz $f$,
\begin{align*}
\biggl| \int f(x) h_\ell(x)\,dx \biggr| = \bigl| \tr \bigl( R_0 f R_0 q (R_0q)^{\ell-1} \bigr)\bigr| \leq \| R_0 f R_0\|_{\op} \| q R_0 q \|_{\I_1} \| R_0 q\|_{\op}^{\ell-2},
\end{align*}
which we shall now use to estimate $h_\ell$ via duality. First, we apply the following elementary estimates
\begin{gather*}
\| R_0 f R_0\|_{\op} \lesssim \|f\|_{H^{-2}_\kappa},\quad
 \| q R_0 q \|_{\I_1} = \tfrac1{2\kappa} \|q\|_{L^2}^2, \qtq{and} \| R_0 q\|_{\op} \lesssim \kappa^{-3/2}\|q\|_{L^2},
\end{gather*}
to deduce that the series \eqref{A6} converges uniformly in $H^2_\kappa$ on $B_\delta$, for $\delta>0$ sufficiently small.  Indeed, this analysis readily yields also that
\begin{align*}
\bigl\| g-\tfrac1{2\kappa} - h_1 \bigr\|_{H^{2}_\kappa} \lesssim \kappa^{-1} \|q\|_{L^2}^2 \qtq{and}
    \bigl\| dg - dh_1 \bigr\|_{L^2\to H^{2}_\kappa} \lesssim \kappa^{-1} \delta .
\end{align*}
Combined with \eqref{E:h1}, these two estimates allow one to readily complete the proof of part (a) by mimicking the arguments used to prove Proposition~\ref{P:diffeo}; compare, in particular, \eqref{g H1 bound} and \eqref{inverse input}.

It remains to justify \eqref{A5tozero}.  In view of \eqref{E:h1conv} and \eqref{E:h2conv}, this follows from
\begin{align*}
\bigl\| g-\tfrac1{2\kappa} - h_1 - h_2 \bigr\|_{H^{-2}} \lesssim \kappa^{-13/2} \|q\|_{L^2}^3,
\end{align*}
which follows in turn by incorporating the elementary bound
\begin{gather*}
\| R_0 f R_0\|_{\op} \lesssim \kappa^{-4} \|f\|_{L^\infty} \lesssim \kappa^{-4} \|f\|_{H^{2}}
\end{gather*}
into the preceding duality analysis.
\end{proof}

Our next task is to decide on a proper replacement for $H_\kappa$.  Proceeding formally, one is lead to the asymptotic expansion
$$
\alpha(\kappa;q) = \tfrac{1}{4\kappa^3} P(q)  - \tfrac{1}{16\kappa^5} H_\text{KdV}(q) + \tfrac{1}{64\kappa^7} H_\text{5th}(q) + O(\kappa^{-9})
$$
and thus to a natural proposal for renormalized Hamiltonians, namely,
\begin{align*}
H^\text{5th}_\kappa (q) := 64\kappa^7 \alpha(\kappa;q) - 16 \kappa^4 P(q) + 4\kappa^2 H_\text{KdV}(q) = 4\kappa^2\bigl[ H_\text{KdV}(q) - H_\kappa(q) \bigr].
\end{align*}

We will now discuss analogues of Propositions~\ref{L:5.2}, \ref{P:H kappa}, and~\ref{P:equi}:

\begin{prop} There exists $\delta>0$ so that the following hold:\\
\emph{(a)} The flow generated by $H^\text{\rm 5th}_\kappa$ is
\begin{align}\label{5th q}
\tfrac{d\ }{dt}  \, q(x) =  -64\kappa^7 g'(x;\kappa,q) - 16 \kappa^4 q'(x) + 4\kappa^2 \bigl[ -q'''(x) + 6 q(x)q'(x) \bigr],
\end{align}
which is globally well-posed on $B_\delta$ for any $\kappa\geq1;$ moreover, under this flow,
\begin{align}\label{5th kappa 1/g}
\tfrac{d\ }{dt}  \, \tfrac{1}{2g(x;\vk)} &= \tfrac{16\kappa^7}{\kappa^2-\vk^2} \Bigl( \tfrac{g(x;\kappa)}{g(x;\vk)} \Bigr)' - 16 \kappa^4 \Bigl( \tfrac{1}{2g(x;\vk)} \Bigr)'
    + 4\kappa^2 \Bigl( \tfrac{q(x)-2\vk^2}{g(x;\vk)}  \Bigr)'
\end{align}
for any $1\leq \vk < \kappa$ and $\alpha(q;\vk)$ is conserved.\\
\emph{(b)} Consider now the $H_\text{\rm 5th}$ flow \eqref{5th}, which is globally well-posed on $H^2(\R)$.  For initial data in $B_\delta\cap H^2$ and any $\vk\geq 1$,
\begin{align}\label{5th 1/g}
\tfrac{d\ }{dt}  \, \tfrac{1}{2g(x;\vk,q(t))} &= \Bigl(\tfrac{-q''(t,x)+3q(t,x)^2-4\vk^2q(t,x)+8\vk^4}{g(x;\vk,q(t))} \Bigr)';
\end{align}
moreover, $H^\text{\rm 5th}_\vk$ is conserved.\\
\emph{(c)} Let $Q\subseteq B_\delta\cap H^2$ be $L^2$-equicontinuous and let
\begin{align}\label{Q star 5th}
Q^* = \bigl\{ e^{J\nabla(t H_\text{\rm 5th} + s H^\text{\rm 5th}_\kappa)} q : q\in Q,\ t,s\in \R,\text{ and } \kappa\geq 1 \bigr\}.
\end{align}
Then $Q^*$ is $L^2$-equicontinuous and so convergence in \eqref{A5tozero} is uniform on $Q^*$.
\end{prop}

\begin{proof}
As $H^\text{5th}_\kappa = 4\kappa^2[ H_\text{KdV} - H_\kappa ]$, equation \eqref{5th q} follows from \eqref{H kappa flow q}, while \eqref{5th kappa 1/g} follows directly from Propositions~\ref{L:5.2} and~\ref{P:H kappa}.  Well-posedness and commutativity of these constituent flows (together with conservation of $P$) immediately guarantees the well-posedness of this flow in the $H^{-1}$ topology on $B_\delta$.  Continuity in the $L^2$-topology will follow when we prove part (c) below.

The conservation of $\alpha(q;\vk)$ under the $H^\text{5th}_\kappa$ follows from its conservation under the $H_\text{KdV}$ and $H_\kappa$ flows.  It may also be deduced directly from \eqref{5th q} and \eqref{5th kappa 1/g}.

As noted earlier, the well-posedness of \eqref{5th} on $H^2(\R)$ is a result of \cite{MR3096990,MR3301874}.  To verify \eqref{5th 1/g}, we first observe that \eqref{5th} can be rewritten as
$$
\ddt q = (-\partial_x^3 + 2 q\partial_x  + 2\partial_x q + 4\vk^2\partial_x) (-q''+3q^2-4\vk^2q) + 16\vk^4 q' ,
$$
which is little more that the biHamiltonian relation for the KdV hierarchy.  Applying Lemma~\ref{L:G ibp}  just as in the proof of Proposition~\ref{L:5.2}, we deduce that
\begin{align*}
\tfrac{d\ }{dt}  \, g &= -2(-q''+3q^2-4\vk^2q)'g + 2(-q''+3q^2-4\vk^2q)g' + 16\vk^4 g'
\end{align*}
where we abbreviate  $g=g(x;\vk,q)$.  This then leads to \eqref{5th 1/g} via the chain rule.

To verify that the $H_\text{\rm 5th}$ flow conserves $H^\text{5th}_\vk$, it suffices to verify that it conserves $P(q)$, $H_\text{KdV}(q)$, and $\alpha(q;\vk)$.  The first two follow from elementary calculations, while the conservation of $\alpha$ follows from \eqref{5th} and \eqref{5th 1/g}.

The conservation of $\alpha$ holds the key to proving part (c); we simply need to modify Lemma~\ref{L:equi 2} in the following way (cf. \cite[Proposition~3.6]{KVZ}):  In view of \eqref{R I2},
\begin{align*}
\biggl| 4\vk^3\alpha(q;\vk) - \tfrac12\int \frac{4\vk^2 |\hat q(\xi)|^2\,d\xi}{\xi^2+4\vk^2} \biggr|
    &\leq 4\vk^3 \sum_{\ell=3}^\infty \bigl\| \sqrt{R_0} q \sqrt{R_0} \bigr\|^{\ell}_{\op}
        \lesssim  \vk^{-3/2}
\end{align*}
uniformly for $q\in B_\delta$ and $\vk\geq 1$, provided $\delta$ is sufficiently small.  Thus
$$
\text{$Q^*$ is $L^2$-equicontinuous} \iff [P(q)-4\vk^3\alpha(q;\vk)] \to 0\ \text{uniformly on $Q^*$ as $\vk\to\infty$}
$$
and (by conservation of $\alpha$) this holds if and only if $Q$ is $L^2$-equicontinuous.
\end{proof}

\begin{proof}[Proof of Theorem~\ref{T:5th}]
Reviewing what has gone before, we see that it suffices to show that for any $L^2$-equicontinuous set $Q\subseteq B_\delta\cap H^2$,
\begin{align}\label{destination A1}
\lim_{\kappa\to\infty} \sup_{q\in Q^*} \ \sup_{|t|\leq T}\ \| e^{tJ\nabla (H_\text{\rm 5th} - H_\kappa^\text{\rm 5th})} q - q \|_{L^2} =0.
\end{align}

With a view to further reduction, let us define
$$
g(t,x) = g(x;\vk,q(t)) \qtq{where} q(t) = e^{tJ\nabla(H_\text{\rm 5th} - H^\text{\rm 5th}_\kappa)} q \qtq{and} q\in Q^*.
$$
In view of Proposition~\ref{P:A2}(a), \eqref{translation identity}, and the equicontinuity of $Q^*$, we see that
$$
E:=\Bigl\{ \tfrac{1}{2 g(t,x)} - \vk : q\in Q^* \textrm{ and }t\in\R\Bigr\} \quad\text{is $H^2(\R)$-equicontinuous.}
$$
Thus, by Lemma~\ref{L:equi 1}(ii) and Proposition~\ref{P:A2}(a), we may verify \eqref{destination A1} by showing that
\begin{align}\label{E:5th finale}
\bigl\| \tfrac{d\ }{dt} \bigl(\varkappa - \tfrac{1}{2g(t;\vk)}\bigr) \bigr\|_{H^{-3}} \to 0
\end{align}
as $\kappa\to\infty$ uniformly for $q\in Q^*$ and $t\in\R$.

We now verify \eqref{E:5th finale}.  Combining \eqref{5th kappa 1/g} and~\eqref{5th 1/g}, we deduce (after considerable rearrangement) that
\begin{align*}
\tfrac{d\ }{dt}  \, \tfrac{1}{2g(x;\vk)} &= \Bigl( \tfrac{-q''(t,x)+3q(t,x)^2 - 16\kappa^5(g(x;\kappa)-\frac{1}{2\kappa})-4\kappa^2q}{g(x;\vk)} \Bigr)' \\
    &\qquad  + \vk^2 \Bigl( \tfrac{-4q(t,x) - 16\kappa^3(g(x;\kappa)-\frac{1}{2\kappa})}{g(x;\vk)} \Bigr)' \\
    &\qquad  - \tfrac{\vk^4\kappa^2}{\kappa^2-\vk^2}\Bigl( \tfrac{16\kappa(g(x;\kappa)-\frac{1}{2\kappa})}{g(x;\vk)} \Bigr)' - \tfrac{\vk^6}{\kappa^2-\vk^2}\Bigl( \tfrac{8}{g(x;\vk)} \Bigr)' .
\end{align*}
That this converges to zero in the desired sense then follows trivially from \eqref{A5tozero}.
\end{proof}


\begin{thebibliography}{99}



\bibitem{Appell}
M. Appell,
Sur la transformation des \'equations diff\'erentielles lin\'eaires.
Comptes Rendus \textbf{91} (1880), no. 4, 211--214.

\bibitem{MR0997295}
V. I. Arnol'd,
\emph{Mathematical methods of classical mechanics. Second edition.}
Translated from the Russian by K. Vogtmann and A. Weinstein.  Graduate Texts in Mathematics, \textbf{60}. Springer-Verlag, New York, 1989.

\bibitem{MR0427868}
T. B. Benjamin, J. L. Bona, and J. J. Mahony,
Model equations for long waves in nonlinear dispersive systems.
Philos. Trans. Roy. Soc. London Ser. A \textbf{272} (1972), no. 1220, 47--78.

\bibitem{MR0393887}
J. Bona and R. Scott,
Solutions of the Korteweg-de Vries equation in fractional order Sobolev spaces.
Duke Math. J. \textbf{43} (1976), no. 1, 87--99.

\bibitem{MR0385355}
J. L. Bona and R. Smith,
The initial-value problem for the Korteweg-de Vries equation.
Philos. Trans. Roy. Soc. London Ser. A \textbf{278} (1975), no. 1287, 555--601.

\bibitem{MR1215780}
J. Bourgain,
Fourier transform restriction phenomena for certain lattice subsets and applications to nonlinear evolution equations. II. The KdV-equation.
Geom. Funct. Anal. \textbf{3} (1993), no. 3, 209--262.


\bibitem{Boo}
J. Boussinesq,
Th\'eorie des ondes et des remous qui se propagent le long d'un canal
rectangulaire horizontal, en communiquant au liquide contenu dans
ce canal des vitesses sensiblement pareilles de la surface au fond.
J. Math. Pures et Appl. (2) \textbf{17} (1872), 55--108.


\bibitem{MR3400442}
T. Buckmaster and H. Koch,
The Korteweg--de Vries equation at $H^{-1}$ regularity.
Ann. Inst. H. Poincar\'e Anal. Non Lin\'eaire \textbf{32} (2015), no. 5, 1071--1098.


\bibitem{Christ'05}
M. Christ,
Nonuniqueness of weak solutions of the nonlinear Schr\"odinger equation.
Preprint \texttt{arXiv:math/0503366}.

\bibitem{MR2018661}
M. Christ, J. Colliander, and T. Tao,
Asymptotics, frequency modulation, and low regularity ill-posedness for canonical defocusing equations.
Amer. J. Math. \textbf{125} (2003), no. 6, 1235--1293.



\bibitem{MR0069338} E. A. Coddington and N. Levinson,
\emph{Theory of ordinary differential equations.}
McGraw-Hill Book Company, Inc., New York-Toronto-London, 1955.

\bibitem{MR1969209}
J. Colliander, M. Keel, G. Staffilani, H. Takaoka, and T. Tao,
Sharp global well-posedness for KdV and modified KdV on $\R$ and $\mathbb{T}$.
J. Amer. Math. Soc. \textbf{16} (2003), no. 3, 705--749.

\bibitem{MR2054622}
J. Colliander, M. Keel, G. Staffilani, H. Takaoka, and T. Tao,
Multilinear estimates for periodic KdV equations, and applications.
J. Funct. Anal. \textbf{211} (2004), no. 1, 173--218.

\bibitem{MR2233689}
J. Colliander, M. Keel, G. Staffilani, H. Takaoka, and T. Tao,
Symplectic nonsqueezing of the Korteweg-de Vries flow.
Acta Math. \textbf{195} (2005), 197--252.

\bibitem{MR1329553}
D. G. Crighton,
Applications of KdV.
In ``KdV '95 (Amsterdam, 1995)''.
Acta Appl. Math. \textbf{39} (1995), no. 1--3, 39--67.

\bibitem{MR0875319}
A. Degasperis and P. C. Sabatier,
Extension of the one-dimensional scattering theory, and ambiguities.
Inverse Problems \textbf{3} (1987), no. 1, 73--109.

\bibitem{MR0427869}
B. A. Dubrovin, V. B. Matveev, and S. P. Novikov,
Nonlinear equations of Korteweg-de Vries type, finite-band linear operators and Abelian varieties.
Uspehi Mat. Nauk \textbf{31} (1976), no. 1, 55--136.


\bibitem{MR0403368}
H. Flaschka and D. W. McLaughlin,
Canonically conjugate variables for the Korteweg-de Vries equation and the Toda lattice with periodic boundary conditions.
Progr. Theoret. Phys. \textbf{55} (1976), no. 2, 438--456.


\bibitem{GGKM}
C. S. Gardner, J. M. Greene, M. D. Kruskal, and R. M. Miura,
Method for solving the Korteweg-de Vries equation.
Phys. Rev. Lett. \textbf{19} (1967), no. 19, 1095--1097.

\bibitem{MR0749109}
J. Garnett and E. Trubowitz,
Gaps and bands of one-dimensional periodic Schr\"odinger operators.
Comment. Math. Helv. \textbf{59} (1984), no. 2, 258--312.

\bibitem{MR2653659}
A. Gr\"unrock,
On the hierarchies of higher order mKdV and KdV equations.
Cent. Eur. J. Math. \textbf{8} (2010), no. 3, 500--536.

\bibitem{MR2531556}
Z. Guo,
Global well-posedness of Korteweg-de Vries equation in $H^{-3/4}(\R)$.
J. Math. Pures Appl. (9) \textbf{91} (2009), no. 6, 583--597.

\bibitem{MR3096990}
Z. Guo, C. Kwak, and S. Kwon,
Rough solutions of the fifth-order KdV equations.
J. Funct. Anal. \textbf{265} (2013), no. 11, 2791--2829.

\bibitem{Hilbert}
D. Hilbert,
Grundz\"uge einer allgemeinen Theorie der linearen Integralgleichungen (Erste Mitteilung).
Nachr. Ges. Wiss. G\"ottingen (1904), 49--91.


\bibitem{MR0044404}
R. Jost and A. Pais,
On the scattering of a particle by a static potential.
Physical Rev. (2) \textbf{82}, (1951), no. 6, 840--851.

\bibitem{MR2179653}
T. Kappeler, C. M\"ohr, and P. Topalov,
Birkhoff coordinates for KdV on phase spaces of distributions.
Selecta Math. (N.S.) \textbf{11} (2005), no. 1, 37--98.

\bibitem{MR3739929}
T. Kappeler and J.-C. Molnar,
On the wellposedness of the KdV/KdV2 equations and their frequency maps.
Ann. Inst. H. Poincar\'e Anal. Non Lin\'eaire \textbf{35} (2018), no. 1, 101--160.

\bibitem{MR2189502}
T. Kappeler, P. Perry, M. Shubin, and P. Topalov,
The Miura map on the line.
Int. Math. Res. Not. 2005, no. 50, 3091--3133.

\bibitem{MR1997070}
T. Kappeler and J. P\"oschel,
\emph{KdV \& KAM.}
Ergebnisse der Mathematik und ihrer Grenzgebiete, \textbf{45}. Springer-Verlag, Berlin, 2003.



\bibitem{MR2267286}
T. Kappeler and P. Topalov,
Global wellposedness of KdV in $H^{-1}(\mathbb{T},\mathbb{R})$.
Duke Math. J. \textbf{135} (2006), no. 2, 327--360.

\bibitem{MR0190801}
T. Kato,
Wave operators and similarity for some non-selfadjoint operators.
Math. Ann. \textbf{162} (1965/1966), 258--279.

\bibitem{MR0234314}
T. Kato,
Smooth operators and commutators.
Studia Math. \textbf{31} (1968), 535--546.

\bibitem{MR0407477}
T. Kato,
Quasi-linear equations of evolution, with applications to partial differential equations.
In ``Spectral theory and differential equations'' (Proc. Sympos., Dundee, 1974; dedicated to Konrad J\"orgens), pp. 25--70. Lecture Notes in Math., Vol. 448, Springer, Berlin, 1975.

\bibitem{MR0759907}
T. Kato,
On the Cauchy problem for the (generalized) Korteweg-de Vries equation. Studies in applied mathematics, 93--128,
Adv. Math. Suppl. Stud., 8, Academic Press, New York, 1983.

\bibitem{MR3301874}
C. Kenig and D. Pilod,
Well-posedness for the fifth-order KdV equation in the energy space.
Trans. Amer. Math. Soc. \textbf{367} (2015), no. 4, 2551--2612.

\bibitem{MR1086966}
C. E. Kenig, G. Ponce, and L. Vega,
Well-posedness of the initial value problem for the Korteweg-de Vries equation.
J. Amer. Math. Soc. \textbf{4} (1991), no. 2, 323--347.


\bibitem{MR1329387}
C. E. Kenig, G. Ponce, and L. Vega,
A bilinear estimate with applications to the KdV equation.
J. Amer. Math. Soc. \textbf{9} (1996), no. 2, 573--603.

\bibitem{MR2501679}
N. Kishimoto,
Well-posedness of the Cauchy problem for the Korteweg-de Vries equation at the critical regularity.
Differential Integral Equations \textbf{22} (2009), no. 5--6, 447--464.

\bibitem{MR2310217}
R. Killip,
Spectral theory via sum rules. In \emph{Spectral theory and mathematical physics: a Festschrift in honor of Barry Simon's 60th birthday}, 907--930,
Proc. Sympos. Pure Math., \textbf{76}, Part 2, Amer. Math. Soc., Providence, RI, 2007.

\bibitem{MR2552106}
R. Killip and B. Simon,
Sum rules and spectral measures of Schr\"odinger operators with $L^2$ potentials.
Ann. of Math. (2) \textbf{170} (2009), no. 2, 739--782.

\bibitem{KVZ}
R. Killip, M. Visan, and X. Zhang,
Low regularity conservation laws for integrable PDE.
Preprint \texttt{arXiv:1708.05362}.

\bibitem{MR2138138}
A. Kiselev,
Imbedded singular continuous spectrum for Schr\"odinger operators.
J. Amer. Math. Soc. \textbf{18} (2005), no. 3, 571--603.



\bibitem{KT}
H. Koch and D. Tataru,
Conserved energies for the cubic NLS in 1-d.
Preprint \texttt{arXiv:1607.02534}.

\bibitem{KdV1895}
D. J. Korteweg and G. de Vries,
On the change of form of long waves advancing in a rectangular canal, and on a new type of long stationary waves. Philosophical Magazine, \textbf{39} (1895), no. 240, 422--443.

\bibitem{MR0235310}
P. D. Lax,
Integrals of nonlinear equations of evolution and solitary waves.
Comm. Pure Appl. Math. \textbf{21} (1968), 467--490.


\bibitem{MR3292346}
B. Liu,
A priori bounds for KdV equation below $H^{-3/4}$.
J. Funct. Anal. \textbf{268} (2015), no. 3, 501--554.

\bibitem{MR0559928}
W. Magnus and S. Winkler,
\emph{Hill's equation.}
Corrected reprint of the 1966 edition. Dover Publications, Inc., New York, 1979.

\bibitem{MR0897106}
V. A. Marchenko,
\emph{Sturm-Liouville operators and applications.}
Translated from the Russian by A. Iacob. Operator Theory: Advances and Applications, \textbf{22}. Birkh\"auser Verlag, Basel, 1986.

\bibitem{MR0409965}
V. A. Mar\v{c}enko and I. V. Ostrovski\u\i,
A characterization of the spectrum of the Hill operator.
Math. USSR-Sb. \textbf{26} (1975), no. 4, 493--554.

\bibitem{MR0427731}
H. P. McKean and E. Trubowitz,
Hill's operator and hyperelliptic function theory in the presence of infinitely many branch points.
Comm. Pure Appl. Math. \textbf{29} (1976), no. 2, 143--226.

\bibitem{MR0397076}
H. P. McKean and P. van Moerbeke,
The spectrum of Hill's equation.
Invent. Math. \textbf{30} (1975), no. 3, 217--274.

\bibitem{MR0792566}
A. Melin,
Operator methods for inverse scattering on the real line.
Comm. Partial Differential Equations \textbf{10} (1985), no. 7, 677--766.

\bibitem{MR0252825}
R. Miura,
Korteweg-de Vries equation and generalizations. I. A remarkable explicit nonlinear transformation.
J. Mathematical Phys. \textbf{9} (1968), no. 8, 1202--1204.


\bibitem{MR0252826}
R. M. Miura, C. S. Gardner, and M. D. Kruskal,
Korteweg-de Vries equation and generalizations. II. Existence of conservation laws and constants of motion.
J. Mathematical Phys. \textbf{9} (1968), no. 8, 1204--1209.

\bibitem{MR2830706}
L. Molinet,
A note on ill posedness for the KdV equation.
Differential Integral Equations \textbf{24} (2011), no. 7-8, 759--765.

\bibitem{MR2927357}
L. Molinet,
Sharp ill-posedness results for the KdV and mKdV equations on the torus.
Adv. Math. \textbf{230} (2012), no. 4--6, 1895--1930.






\bibitem{MR0493419}
M. Reed and B. Simon,
\emph{Methods of modern mathematical physics. I. Functional analysis.} Academic Press, New York-London, 1972.

\bibitem{MR0493420}
M. Reed and B. Simon,
\emph{Methods of modern mathematical physics. II. Fourier analysis, self-adjointness.} Academic Press, New York-London, 1975.

\bibitem{MR0493421}
M. Reed and B. Simon,
\emph{Methods of modern mathematical physics. IV. Analysis of operators.} Academic Press, New York-London, 1978.

\bibitem{MR2683250}
A. Rybkin,
\emph{Regularized perturbation determinants and KdV conservation laws for irregular initial profiles.}
Topics in operator theory. Volume 2. Systems and mathematical physics, 427--444, Oper. Theory Adv. Appl., \textbf{203}, Birkh\"auser Verlag, Basel, 2010.



\bibitem{MR0454425} 
J. C. Saut and R. Temam,
Remarks on the Korteweg-de Vries equation.
Israel J. Math. \textbf{24} (1976), no. 1, 78--87.

\bibitem{MR2154153}
B. Simon,
\emph{Trace ideals and their applications.}
Second edition. Mathematical Surveys and Monographs, 120. American Mathematical Society, Providence, RI, 2005.

\bibitem{Sj}
A. Sj\"oberg,
On the Korteweg-de Vries equation. Report Dept of Computer Science Upsala University, 1967.

\bibitem{MR0410135}
A. Sj\"oberg,
On the Korteweg-de Vries equation: existence and uniqueness.
J. Math. Anal. Appl. \textbf{29} (1970), 569--579.




\bibitem{MR0261183}
R. Temam,
Sur un probl\`eme non lin\'eaire.
J. Math. Pures Appl. (9) \textbf{48} (1969), 159--172.

\bibitem{MR0312097}
M. Tsutsumi and T. Mukasa,
Parabolic regularizations for the generalized Korteweg-de Vries equation.
Funkcial. Ekvac. \textbf{14} (1971), 89--110.

\bibitem{MR0990865}
Y. Tsutsumi,
The Cauchy problem for the Korteweg-de Vries equation with measures as initial data.
SIAM J. Math. Anal. \textbf{20} (1989), no. 3, 582--588.

\bibitem{MR0501141}
J. Vey,
Sur certains syst\`emes dynamiques s\'eparables.
Amer. J. Math. \textbf{100} (1978), no. 3, 591--614.

\bibitem{KruskalZubusky}
N. J. Zabusky and M. D. Kruskal,
Interaction of ``solitons" in a collisionless plasma and the recurrence of initial states.
Phys. Rev. Lett. \textbf{15}, no. 15, (1965) 240--243.

\bibitem{MR0303132}
V. E. Zaharov and L. D. Faddeev,
The Korteweg-de Vries equation is a completely integrable Hamiltonian system.
Funkcional. Anal. i Prilo\v{z}en. \textbf{5} (1971), no. 4, 18--27.

\bibitem{MR2150385}
N. T. Zung,
Convergence versus integrability in Birkhoff normal form.
Ann. of Math. (2) \textbf{161} (2005), no. 1, 141--156.

\end{thebibliography}
\end{document}